\documentclass[a4paper,11pt,english,reqno]{amsart}
\usepackage{amsmath,amssymb,amsfonts,dsfont,amsthm,upgreek,bm,mathtools}    
\usepackage[utf8]{inputenc}
\usepackage[english]{babel}
\usepackage[T1]{fontenc} 
\usepackage{graphicx,xcolor,pgfkeys}
\usepackage{float}
\usepackage[font=small,labelfont=bf]{caption}
\usepackage[a4paper, margin=2.4cm]{geometry}
\usepackage{tikz,pgfplots}
\usepackage{tkz-fct}
\usepgfplotslibrary{polar}
\usetikzlibrary{patterns, arrows}

\usepackage{enumerate}
\usepackage{enumitem}
\usepackage[normalem]{ulem}
\newcommand\redsout{\bgroup\markoverwith{\textcolor{red}{\rule[0.5ex]{2pt}{1pt}}}\ULon}
\newcommand{\stkout}[1]{\ifmmode\text{\redsout{\ensuremath{#1}}}\else\redsout{#1}\fi}
\usepackage{scalerel,stackengine}
\stackMath
\newcommand\reallywidecheck[1]{%
\savestack{\tmpbox}{\stretchto{%
  \scaleto{%
    \scalerel*[\widthof{\ensuremath{#1}}]{\kern-.6pt\bigwedge\kern-.6pt}%
    {\rule[-\textheight/2]{1ex}{\textheight}}
  }{\textheight}%
}{0.6ex}}%
\stackon[1pt]{#1}{\scalebox{-0.8}{\tmpbox}}%
}

\newcommand{\prob}{\mathbf{P}}
\newcommand{\erw}{\mathds{E}}
\newcommand{\e}{\mathrm{e}}

\newcommand{\comp}{\mathrm{comp}}

\newcommand{\spt}{\mathrm{supp}}

\newcommand{\WF}{\mathrm{WF}}
\newcommand{\supp}{\mathrm{supp\,}}

\newcommand{\dist}{\mathrm{dist\,}}

\newcommand{\Op}{\mathrm{Op}}
\newcommand{\mO}{\mathcal{O}}

\newcommand{\R}{\mathbf{R}}

\newcommand{\N}{\mathbf{N}}
\newcommand{\Z}{\mathbf{Z}}

\newcommand{\wt}{\widetilde}

\newcommand{\wh}{\widehat}

\newcommand{\bA}{\boldsymbol{A}}
\newcommand{\bpsi}{\boldsymbol{\psi}}
\newcommand{\bB}{\boldsymbol{B}}

\newcommand{\wit}{\widetilde}

\newcommand{\E}{\mathbb{E}}

\newcommand{\esupp}{\mathrm{ess}\mbox{-}\mathrm{supp}\,}

\newcommand{\cS}{\mathcal{S}}
\newcommand{\cT}{\mathcal{T}}

\newcommand{\cI}{\mathcal{I}}

\newcommand{\cK}{\mathcal{K}}

\newcommand{\cU}{\mathcal{U}}

\renewcommand{\geq}{\geqslant}

\renewcommand{\leq}{\leqslant}

\newcommand{\be}{\boldsymbol{e}}
\newcommand{\bx}{\boldsymbol{x}}
\newcommand{\bxi}{\boldsymbol{\xi}}
\newcommand{\bchi}{\boldsymbol{\chi}}

\newcommand{\bU}{\boldsymbol{U}}

\newcommand{\dd}{\mathrm{d}}
\newtheorem{thm}{Theorem}
\newtheorem{corollary}[thm]{Corollary}

\newtheorem{prop}[thm]{Proposition}
\newtheorem{lem}[thm]{Lemma}

\newtheorem{definition}[thm]{Definition}
\newtheorem{rem}[thm]{Remark}

\newtheorem{ex}[thm]{Example}
\newtheorem{hypo}[thm]{Hypothesis}

\numberwithin{equation}{section}
\numberwithin{thm}{section}

\usepackage{hyperref}
\hypersetup{pdfborder=0 0 0, 
	    colorlinks=true,
	    citecolor=blue,
	    linkcolor=blue,
	    urlcolor=blue,
	    pdfauthor={Maxime Ingremeau, Martin Vogel}
	   }

\def\ora{\textcolor{black}}

\tikzset{
xmin/.store in=\xmin, xmin/.default=-3, xmin=-3,
xmax/.store in=\xmax, xmax/.default=3, xmax=3,
ymin/.store in=\ymin, ymin/.default=-3, ymin=-3,
ymax/.store in=\ymax, ymax/.default=3, ymax=3,
}

\title[Improved $L^\infty$ bounds under random perturbations in negative curvature]
{Improved $L^\infty$ bounds for eigenfunctions under random perturbations in negative curvature}
\author{Maxime Ingremeau}
\address[Maxime Ingremeau]{Laboratoire J. A. Dieudonn\'e
UMR CNRS-UNS 351, 
Universit\'e C\^ote d'Azur, 
Parc Valrose, 
06108 NICE Cedex 2, France.}
\email{Maxime.Ingremeau@univ-cotedazur.fr}
\author{Martin Vogel}
\address[Martin Vogel]{Institut de Recherche Math{\'e}matique Avanc{\'e}e - UMR 7501, 
Universit{\'e} de Strasbourg et CNRS, 7 rue René-Descartes, 67084 Strasbourg Cedex, France.}
\email{vogel@math.unistra.fr}
\date{\today}

\usepackage[most]{tcolorbox}
\begin{document}
\begin{abstract}
It has been known since the work of Avakumov\'ic, Hörmander and Levitan 
that, on any compact smooth Riemannian manifold, if $-\Delta_g \psi_\lambda = \lambda 
\psi_\lambda$, then $\|\psi_\lambda\|_{L^\infty} \leq C 
\lambda^{\frac{d-1}{4}} \|\psi_\lambda\|_{L^2}$. It is believed that, 
on manifolds of negative curvature, such a bound can be largely improved; however, only logarithmic improvements in $\lambda$ have been obtained so far.
In the present paper,  we obtain polynomial improvements over the previous bound in a generic setting, by adding a small random pseudodifferential perturbation to the Laplace-Beltrami operator.
\end{abstract}
\maketitle
\tableofcontents
\section{Introduction}
On a smooth compact Riemannian manifold $(X,g)$, the high energy (or semiclassical) properties 
of the eigenfunctions of the Laplace-Beltrami operator $\Delta_g$ are known to be 
strongly related to the classical dynamics on the manifold, i.e. the 
dynamical properties of the geodesic flow. As such, a fundamental intuition 
in the field of Quantum Chaos is that, if the geodesic flow 
is chaotic (as is the case, for instance, o\ora{f} a manifold of negative curvature), then the 
high energy eigenfunctions should be delocalized. There are several ways of 
describing the delocalization properties of 
eigenfunctions \ora{of 
the Laplace-Beltrami operator (i.e., functions 
$\psi_\lambda$ with $-\Delta_g \psi_\lambda = \lambda \psi_\lambda$)}: 
one is through the use of semiclassical measures, i.e. weak limits of the 
probability measures $|\psi_\lambda(x)|^2 \dd x$ (and their lifts to $S^*X$) 
as $\lambda\to\infty$; 
another is by studying the $L^p$ norms of $\psi_\lambda$, and in particular, 
their $L^\infty$ norms.
\par
It is has been known since the work of Avakumov\'ic, Hörmander and Levitan 
\cite{Ava, Hor, Lev} that, on any compact smooth Riemannian manifold, 
if \ora{$\psi_\lambda\in L^2$ satisfies} $-\Delta_g \psi_\lambda = \lambda \psi_\lambda$,  then
\begin{equation}\label{eq:Hor}
\|\psi_\lambda\|_{L^\infty} \leq 
	C \lambda^{\frac{d-1}{4}} \|\psi_\lambda\|_{L^2},
\end{equation}
and these bounds are saturated by zonal spherical harmonics. More generally, 
sharp bounds for the other $L^p$ norms of eigenfunctions were obtained by 
Sogge in \cite{Sogge}.
\par
It is believed that, on manifolds of negative curvature, such bounds can be 
largely improved, thus reflecting the aforementioned delocalization of 
eigenfunctions. For instance, on surfaces of negative curvature, it was 
conjectured by Sarnak \cite{Sarnak} that $\|\psi_\lambda\|_{L^\infty} = 
\mO(\lambda^{\varepsilon}\|\psi_\lambda\|_{L^2})$ for any $\varepsilon>0$. However, on negatively 
curved manifolds of higher dimension, such a bound cannot hold, since 
\cite{RuSa} gave examples of eigenfunctions whose $L^\infty$ norm grows 
polynomially.
\par
While the study of semiclassical measures of Laplace eigenfunctions on 
manifolds of negative curvature has enjoyed major advances in the past 
years, see for instance \cite{Anan,AN,DyJi,DyJiNo}, much less 
is known for the $L^\infty$ norms. B\'erard \cite{Berard} proved a 
logarithmic improvement on \eqref{eq:Hor} on manifolds of negative curvature, 
i.e. 
\begin{equation}\label{eq:Berard}
\|\psi_\lambda\|_{L^\infty} 
	\leq C \frac{\lambda^{\frac{d-1}{4}}}{\sqrt{\log \lambda}} 
		\|\psi_\lambda\|_{L^2}.
\end{equation}
In the special setting of certain arithmetic surfaces Iwaniec and Sarnak 
\cite{IwSar} obtained \ora{small} polynomial improvements on the estimate \eqref{eq:Hor}, \ora{of the form $\|\psi_\lambda\|_{L^\infty} 
	\leq C_\varepsilon \lambda^{\frac{5}{24}+ \varepsilon} 
		\|\psi_\lambda\|_{L^2}$ for all $\varepsilon>0$.} 
\par
Recently, the estimate \eqref{eq:Berard} was generalized to milder 
dynamical conditions than negative curvature by Bonthonneau \cite{Meeec} 
and Canzani and Galkowski \cite{CanGal2}. Similar logarithmic improvements 
were obtained by several authors for other $L^p$ norms, see 
\cite{HaTa,HeRi,BlaSog,CanGal}. However, no polynomial 
improvements over \eqref{eq:Hor} were obtained in non-arithmetic settings.
\\
\par
It is believed that estimates as (\ref{eq:Hor}) or (\ref{eq:Berard}) 
could more easily be improved in a \emph{generic} setting, for instance by 
changing generically the metric or adding a generic potential. This is the 
approach of the works by Eswarathasan and Toth \cite{EswTot} and 
Canzani, Jakobson and Toth \cite{CaJaTo}. Yet, in these papers, the authors 
obtain bounds on the quantity $\psi_\lambda(x)$ averaged over the perturbation,  
hence, such results do not give information about genuine eigenfunctions of 
a generically perturbed operator. 
\\
\par
The present paper deals with $L^\infty$ norms of eigenfunctions under 
generic perturbation.  Namely,  we consider a manifold of negative 
curvature and build random perturbations of the Laplace-Beltrami operator 
by adding a small random pseudo-differential operator with a symbol 
rapidly oscillating on a mesoscopic scale. Roughly speaking, our main 
result given in Theorem \ref{th:Main} below states that 
\\
\par
\emph{With overwhelming probability, we have a polynomial improvement 
over (\ref{eq:Hor}) for the eigenfunctions of the perturbed operator. 
}
\\
\par
The idea of adding a small (subprincipal) rapidly oscillating pseudo-differential 
perturbation to the Laplacian to prove better quantum chaotic properties was 
first introduced by the authors in \cite{IngVog}, and the present paper 
largely builds upon the techniques developed in \cite{IngVog}.
\section{Presentation of the result}

Let $(X,g)$ be a compact smooth Riemannian manifold of negative sectional 
curvature, connected and without boundary, and let $0<\mu_1< \mu_2$.
Let $\Delta_g$ denote the Laplace-Beltrami operator on $X$. 
\par
In the sequel, we will adopt semiclassical notations: let $h\in(0,1]$ 
be a small 
parameter (corresponding to a multiple of $\lambda^{-1/2}$ in the introduction), 
and let $\psi_h \in L^2(\R^d)$ be a normalized eigenfunction of $-h^2\Delta_g$, 
satisfying
\begin{equation}\label{eq:EIgenfunction}
-h^2\Delta_g\psi_h = E_h \psi_h,
\end{equation}
for some $E_h \in (\mu_1, \mu_2)$.
\par
The bound (\ref{eq:Hor}) can be rephrased as
\begin{equation}\label{eq:BoundLinf}
\|\psi_h\|_{L^\infty} = \mO\!\left(h^{\frac{1-d}{2}}\right).
\end{equation}
Our aim in this article will be to show that (\ref{eq:BoundLinf}) 
can be improved if a small generic perturbation is added to $-h^2\Delta_g$.
\subsection{Simple random perturbations}
Before presenting our principal result in full generality 
we begin by presenting it in a simple setting. 

We first consider the following small random perturbation 
of the semiclassical Laplacian
\begin{equation}\label{eq:SchroedingerOp_n1}
	P_h := -\frac{h^2}{2} \Delta_g + h^\alpha Q_\omega, 
	\quad \alpha \in ]0,1].
\end{equation}
Here, the random perturbation\footnote{The model of random 
perturbations we consider in \eqref{eq:SchroedingerOp_n1} 
can be seen as a 
pseudo-differential analogue of the Anderson model \cite{An58} 
from the theory of random Schrödinger operators.} $Q_\omega$ is a selfadjoint 
pseudo-differential operator given by the quantization of 
\begin{equation}\label{eq:randomSymbol}
	q_\omega(\rho) 
	:= 
	\sum_{j\in J_h} \omega_j \,q_j(\rho),
\end{equation}
where the index set $J_h$ is of cardinality $\mO(h^{-M})$ for some 
$M>0$, $\{\omega_j\}_{j\in J_h}$ is a family of \emph{independent and identically 
distributed (iid)} random variables satisfying the following 
\begin{hypo}[Hypotheses on the random variables]\label{HypPot2a}
	We suppose that the iid random variables 
	$(\omega_j)_{j\in J_h}$ are real-valued and satisfy the following assumptions:
	\begin{enumerate}
		\item We suppose that the random variables $(\omega_j)_{j\in J_h}$, \ora{ defined on a probability space $(\Omega, \prob)$,} 
		  have a common distribution with a compactly supported density 
		  $m\in C^2_c(\R;[0,+\infty[)$ with respect to the Lebesgue measure.
	\item $\mathrm{Var}(\omega_j) >0$.
	\end{enumerate}
\end{hypo}
The \emph{single-site potentials} $q_j$ are constructed 
as follows: Let $\beta \in ]0,1/2[$. We cover the energy shell 
\begin{equation}\label{eq:EnergyShell}
	\mathcal{E}_{[\mu_1,\mu_2]} 
		:= \{(x,\xi)\in T^*X ; |\xi|^2\in [\mu_1, \mu_2]\}, 
		\quad 0<\mu_1 \leq \mu_2, 
\end{equation}
by geodesic balls $B(\rho_{j,h}, h^\beta)$ of radius $h^\beta$
and centred at $\rho_{j,h}\in \mathcal{E}_{[\mu_1,\mu_2]}$, such 
that each point belongs to at most $C$ balls. We then take 
\begin{equation}\label{eq:simpleSCP}
	q_j 
	=
	\chi \left(
		h^{-\beta} \mathrm{dist}_{T^*X}(\rho_{j,h},\rho)
		\right),
\end{equation}
where $\chi\in C_c^\infty([0, \infty);[0,1])$ takes value $1$ on $[0,1]$. 
Here, $\mathrm{dist}_{T^*X}$ denotes the geodesic distance on $T^*X$ 
equipped with an arbitrary metric $g_0$ as at the beginning of Section 
\ref{sec:Class}. For instance one may take the Sasaki metric on $T^*X$ induced 
by $g$. Notice that the construction of $q_j$ yields that 
$|J_h|=\mO(h^{-\beta d})$ in \eqref{eq:randomSymbol}. 
\par
Note that the random perturbation $Q_\omega$ depends on the two parameters 
$\alpha\in]0,1]$ and $\beta\in[0,1/2[$. We require that 
\begin{equation}\label{eq:AlphBeta}
	0<\beta < \min \left( \frac{\alpha}{2}, 2-2\alpha \right).
\end{equation}
\\
\\
\textbf{Conjugating with noisy quantum dynamics.}
Working under assumption \eqref{eq:AlphBeta}, the operator whose 
eigenfunctions we will study is not 
$P_h$ \eqref{eq:SchroedingerOp_n1} itself, but rather the operator
\begin{equation}\label{eq:RandomOperator_n1}
\widetilde{P}_h
:= e^{-i \frac{t}{h} P_h} (-h^2\Delta_g) e^{i \frac{t}{h} P_h},
\end{equation}
for some fixed $t>0$, independent of $h$. By a slightly exotic version of 
Egorov's theorem \ref{app:prop.Egorov}, we have 
\begin{equation}\label{eq:Egorov1}
\begin{aligned}
\widetilde{P}_h & = e^{-i \frac{t}{h} P_h} P_h e^{i \frac{t}{h} P_h} 
	- h^\alpha  e^{-i \frac{t}{h} P_h} Q_\omega e^{i \frac{t}{h} P_h}\\
&= -h^2\Delta_g + h^\alpha Q_\omega - h^\alpha  e^{-i \frac{t}{h} P_h}Q_\omega e^{i \frac{t}{h} P_h}\\
&=: -h^2\Delta_g + h^\alpha \wt{Q}_\omega.
\end{aligned}
\end{equation}
Here\footnote{See Appendix 
\ref{sec:Appendix} for a brief review of the basic notions and notations 
of the pseudo-differential calculus which we use in this paper.\label{FN:app}}, $\wt{Q}_\omega \in \Psi^\comp_\beta(X)$ has full symbol 
\begin{equation}
	\widetilde{q}_\omega = q_\omega-(\Phi^t_h)^*q_\omega+\mO(h^{1-2\beta}) 
	\in S_\beta^{\comp},
\end{equation}
where $\Phi^t_h$ is the Hamiltonian flow associated with the symbol of $P_h$ (see (\ref{eq:HamiltonianFlow}) below).
The operator $\widetilde{P}_h$ \eqref{eq:RandomOperator_n1} may be seen 
as a perturbation of $-h^2\Delta_g$ by a random pseudo-differential operator 
$h^\alpha\wt{Q}_\beta$ with a symbol in $h^\alpha S^{\comp}_\beta(T^*X)$. 
\par
The operator $\widetilde{P}_h$ has the same eigenvalues as $-h^2\Delta_g$, 
and $\psi_h$ is an eigenfunction of $-h^2\Delta_g$ if and only if 
$\wt{\psi}_h=e^{-i \frac{t}{h} P_h} \psi_h$ is an eigenfunction of 
$\widetilde{P}_h$. Our aim will thus be to obtain improved bounds on 
$\| \wt{\psi}_h\|_{L^\infty}$ with high probability. 
%
\begin{thm}[Simple version]\label{th:Main_smpl}
	Let $0<\mu_1<\mu_2$, let $t>0$, and let $\wt{P}_h$ be as in 
	\eqref{eq:RandomOperator_n1}. There exists $\gamma>0$ such that, 
	for all $h\in (0,1]$, the following holds with probability 
	$\geq 1- \mO(h^\infty)$:
	\par
	If $\wt{\psi}_h$ satisfies $\wt{P}_h \wt{\psi}_h = E_h \wt{\psi}_h$ with 
	$E_h \in (\mu_1, \mu_2 )$, then for any $\varepsilon>0$ there 
	exists $C>0$ such that
	\begin{equation}
	\|\wt{\psi}_h\|_{L^\infty} 
	\leq C h^{\frac{1-d}{2} + \gamma - \varepsilon} \|\wt{\psi}_h\|_{L^2}.
	\end{equation}
\end{thm}
\begin{ex}
	As will be clear from the full version of our main result 
	Theorem \ref{th:Main} below, we have control over how $\gamma$ 
	depends on $\alpha$ and $\beta$ as in \eqref{eq:AlphBeta}. 
	See also Remark \ref{rem:exam1} below for a more precise 
	computation. 
	\par 
	So for instance, when $d=2$, we may take $\alpha = \frac{5}{7}$ and 
	$\beta = \frac{2}{7} - \varepsilon$ for an arbitrarily 
	small $\varepsilon>0$. This then yields (with a new arbitrarily 
	small but fixed $\varepsilon>0$)
	\begin{equation}
		\|\wt{\psi}_h\|_{L^\infty} 
		\leq C h^{-\frac{5}{14}- \varepsilon} \|\wt{\psi}_h\|_{L^2}.
	\end{equation}
\end{ex}
	
%
\subsection{More general random perturbations}
Our result holds when the operator $P_h$ in (\ref{eq:SchroedingerOp_n1}) 
is replaced with more general perturbations of $-h^2\Delta$: let us now 
present the precise assumptions we need. We set 
%
%
\begin{equation}\label{eq:SchroedingerOp}
	P_h^\delta := -\frac{h^2}{2} \Delta_g + \delta Q_\omega, 
\end{equation}
where $0\leq \delta = \delta(h) \ll 1$, and where $Q_\omega$ is a selfadjoint 
pseudo-differential operator given by the quantization of a random real-valued 
function $q_\omega : T^*X \longrightarrow \R$ belonging to the symbol 
class 
$S^{-\infty}_\beta(T^*X)$ for some $0<\beta< \frac{1}{2}$. 
More precisely, we 
construct the random symbol $q_\omega$ in the following way: let 
$J_h\subset \N$ be a set of indices of cardinality $|J_h| =\mO(h^{-M})$, for 
some $M>0$, and let $\{q_j\}_{j\in J_h}$ be a family of possibly $h$-dependent 
smooth compactly supported functions on $T^*X$. Let $\omega=\{\omega_j\}_{j\in J_h}$ 
be a sequence of independent and identically distributed (iid) random variables 
satisfying Hypothesis \ref{HypPot2a} above, and set
\begin{equation}\label{eq:randomSymbol}
	q_\omega(\rho) 
	:= 
	\sum_{j\in J_h} \omega_j \,q_j(\rho).
\end{equation}
We make the following additional assumptions on $\{q_j\}_{j\in J_h}$: 
%
%
\begin{hypo}[Hypotheses on the single-site potential]\label{HypPot}
	$\phantom{.}$
\begin{enumerate}[label=\roman*.]
	\item Each $q_j$ is compactly supported, with a support of diameter 
		$\mO(h^\beta)$ uniformly in $j\in J_h$. 
	\item There exists $C>0$, independent of $h$ and $j$, such that for all
	$\rho \in T^*X$, $\rho$ belongs to the support of at most $C$ functions $q_j$.
	\item For any $k\in \N$, there exists $C_k>0$ such that
	\begin{equation}\label{eq:PotentialDer}
		\|q_j\|_{C^k} \leq C_k h^{-\beta k} \quad \forall j\in J_h.
	\end{equation}
	\item Let $0<\mu_1\leq \mu_2$, and let $\varepsilon>0$. 
	There exists $c_0>0$, $\ora{h_0>0}$ such that, for any $T>0$,  \ora{any $0< h \leq h_0$} and any
	$\rho \in\mathcal{E}_{(\mu_1-h^{1-\varepsilon},\mu_2+h^{1-\varepsilon})}$, we have
	\begin{equation}\label{eq:PotentialEverywhere}
	 \sum_{j\in J_h}  \int_0^T  q_j \left(\Phi^t (\rho)\right) dt  
	 \geq c_0 T.
	\end{equation}
	Here, $\mathcal{E}_{(\mu_1,\mu_2)}:= \{(x,\xi)\in T^*X; |\xi|^2 
	\in (\mu_1,\mu_2)\}$ and $\Phi^t : T^*X \longrightarrow T^*X$ denotes the geodesic flow. 
	\end{enumerate}
\end{hypo}
\begin{ex}
	To build such a family of single-site potentials we may use 
	for instance the construction \eqref{eq:simpleSCP}. 
%
\end{ex} 
The symbol of \eqref{eq:SchroedingerOp} is given by 
\begin{equation}\label{eq:Hamiltonian}
	p_\delta(x,\xi) := p(x,\xi;\delta) := \frac{1}{2} |\xi|_x^2+ \delta q_\omega(x, \xi).
\end{equation}
This ($\delta$-dependent and therefore possibly $h$-dependent) Hamiltonian 
induces a Hamiltonian vector field $H_{p_{\delta}}$ which may be defined by the 
pointwise relation $H_{p_{\delta}} \lrcorner\, \sigma= -dp_{\delta}$, where $\sigma$ denotes 
the natural symplectic form on $T^*X$.  Alternatively, one can define $H_{p_{\delta}}$ 
in canonical symplectic coordinates by
\begin{equation}\label{eq:HamiltonianVF} 
	H_{p_{\delta}} 
	=\sum_{k=1}^d \frac{\partial p_{\delta}}{\partial \xi_k}
	\frac{\partial }{\partial x_k}
	-
	\frac{\partial p_{\delta}}{\partial x_k}
	\frac{\partial }{\partial \xi_k}.
\end{equation}
For every $\delta\geq 0$, we denote by 
\begin{equation}\label{eq:HamiltonianFlow} 
	\Phi^{t}_\delta:=\exp(tH_{p_{\delta}})
\end{equation}
the Hamiltonian flow on $T^*X$ generated by the vector field $H_p$.
\subsection{Conjugating with noisy quantum dynamics}\label{SubSec:Conjugate}
Just as in (\ref{eq:RandomOperator_n1}),
the operator whose eigenfunctions we will study in this paper is not 
$P_h^\delta$ itself, but rather the operator
\begin{equation}\label{eq:RandomOperator}
\widetilde{P}_h^\delta
:= e^{-i \frac{t}{h} P_h^\delta} (-h^2\Delta_g) e^{i \frac{t}{h} P_h^\delta},
\end{equation}
for some fixed $t>0$, independent of $h$. As in \eqref{eq:Egorov1}, 
Egorov's theorem \ref{app:prop.Egorov}, gives that  
\begin{equation*}
\widetilde{P}_h^\delta = -h^2\Delta_g + \delta \wt{Q}_\omega.
\end{equation*}
Here, $\wt{Q}_\omega \in \Psi^\comp_\beta(X)$ has full symbol 
\begin{equation}
	\widetilde{q}_\omega = q_\omega-(\Phi^t_\delta)^*q_\omega+\mO(h^{1-2\beta}) 
	\in S_\beta^{\comp}.
\end{equation}
The operator $\widetilde{P}_h^\delta$ may thus be seen as a perturbation of 
$-h^2\Delta_g$ by a random pseudo-differential operator $\delta \wt{Q}_\beta$ 
with a symbol in $\delta S^{\comp}_\beta(T^*X)$. 
\\
\par
The operator $\widetilde{P}_h^\delta$ has the same eigenvalues as $-h^2\Delta_g$, 
and $\psi_h$ is an eigenfunction of $-h^2\Delta_g$ if and only if 
$e^{-i \frac{t}{h} P_h^\delta} \psi_h$ 
is an eigenfunction of $\widetilde{P}_h^\delta$. Our aim will thus be to 
obtain improved bounds on 
\begin{equation*}
\| e^{-i \frac{t}{h} P_h^\delta} \psi_h\|_{L^\infty}
\end{equation*}
with high probability (under appropriate conditions on $\delta$ and $\beta$).
\begin{hypo}\label{Hyp:BetaDelta}
We suppose that there exists $0<\varepsilon_0< \frac{1}{4}$ and $h_0>0$ such that 
for all $0<h\leq h_0$
\begin{equation}\label{eq:CondBeta}
\delta h^{-2\beta - \varepsilon_0} \leq 1,
\end{equation}
\begin{equation}\label{eq:CondBetaDelta}
\delta^2  h^{\beta-2} \geq h^{-\varepsilon_0}.
\end{equation}
\end{hypo}

\begin{rem}
It is natural to consider the case $\delta = h^\alpha$ 
	with $2\beta \leq \alpha$. However, whenever convenient, we will stick with a 
	coupling constant $\delta$ for the sake of generality.
	\par
	Note that when $\delta = h^\alpha$, conditions (\ref{eq:CondBeta}) 
	and (\ref{eq:CondBetaDelta}) rewrite
	$$0<\beta < \min \left( \frac{\alpha}{2}, 2-2\alpha \right).$$
\end{rem}
\subsection{The main result}
Let us take $\Gamma>0$ such that, for $h$ small enough,
\begin{equation}\label{eq:CondGamma}
\begin{aligned}
\delta h^{-2\beta} \leq  h^\Gamma\\
h^{1-\Gamma} \leq \delta h^{\beta/2},
\end{aligned}
\end{equation}
which is possible thanks to Hypothesis \ref{Hyp:BetaDelta}. In particular, 
when $\delta = h^\alpha$, we may take $h\in (0,1]$ and 
\begin{equation}\label{eq:CondGamma.b}
\Gamma <  \min (1- \alpha - \frac{\beta}{2},  \alpha - 2\beta).
\end{equation}
We then set 
\begin{equation}\label{eq:DefGammaPrime}
\Gamma ' := \min \left( \Gamma, \frac{(d-1) \beta}{2} \right).
\end{equation}
\begin{thm}\label{th:Main}
Let $0<\mu_1<\mu_2$, let $t>0$, let $\wt{P}_h^\delta$ be of 
the form (\ref{eq:RandomOperator}). For $h$ small enough, 
the following holds with probability $\geq 1- \mO(h^\infty)$:
\par
If $\psi_h$ satisfies $\wt{P}_h^\delta \psi_h = E_h \psi_h$ with 
$E_h \in (\mu_1, \mu_2 )$, then, for any $\Gamma$ as in 
(\ref{eq:CondGamma}) and any $\varepsilon>0$, taking $\Gamma'$ as in
(\ref{eq:DefGammaPrime}), there exists $C>0$ such that 
\begin{equation}
\|\psi_h\|_{L^\infty} 
\leq C h^{\frac{1-d}{2} + \Gamma' - \varepsilon} \|\psi_h\|_{L^2}.
\end{equation}
\end{thm}
\ora{This theorem deserves several remarks.}

\begin{rem}
Actually, the same result holds 
with the same proof for spectral clusters of the form
\begin{equation*}
	\psi_h = \sum_{\lambda_h\in [E_h, E_h + \mO(h)]}  \psi_{\lambda_h},
\end{equation*}
where $\wt{P}_h^\delta  \psi_{\lambda_h}= \lambda_h \psi_{\lambda_h}$ and 
$E_h\in (\mu_1, \mu_2)$.
\end{rem}
\begin{rem}\label{rem:exam1}
When $d=2$, we may take $\alpha = \frac{5}{7}$ and 
$\beta = \frac{2}{7} - \varepsilon$, leading to $\Gamma' = \frac{1}{7} 
- \varepsilon/2$ for an arbitrarily small $\varepsilon>0$. 
This then yields (with a new arbitrarily small but fixed $\varepsilon>0$)
\begin{equation}
	\|\psi_h\|_{L^\infty} 
	\leq C h^{-\frac{5}{14}- \varepsilon} \|\psi_h\|_{L^2}.
\end{equation}
When $d=3$, the optimal value of $\Gamma'$ which we can obtain from 
(\ref{eq:DefGammaPrime}) is $\Gamma'= \frac{2}{9} - \varepsilon$. 
\end{rem}
\begin{rem}
By interpolation, one can obtain bounds for the other $L^p$ norms. However, 
the value of $\Gamma'$ we can obtain (at least for $d=2$ or $d=3$) does not 
improve Sogge's estimates on the smallest $L^p$ norms, see 
\cite[Theorem 10.10]{Zw12}.
\end{rem}
\begin{rem}
In all the paper,  we make the assumption that $X$ is negatively curved, 
so as to use directly the technical results of \cite{IngVog}. However, 
it is likely that the present argument could be adapted without the 
negative curvature assumption, possibly with further restrictions on 
$\alpha$ and $\beta$, and on the time $t$ in \eqref{eq:RandomOperator}. 
This will be pursued elsewhere.
\end{rem}
\subsection{Organization of the paper}
In section \ref{sec:Class}, we will recall the notations and facts that we 
will use from classical dynamics. Section \ref{Sec:DecompoLag} will be 
devoted to the fact that an eigenfunction of the Laplacian can be decomposed 
as a sum of $\mO(h^{1-d-\varepsilon})$ Lagrangian states; although this fact 
is probably known to experts, we could not find a proof of it anywhere 
in the literature.  In Section 
\ref{sec:Proof}, we will recall a few facts from \cite{IngVog} about 
propagation of Lagrangian states by a noisy quantum evolution, and prove 
Theorem \ref{th:Main}. Finally, in Appendix \ref{sec:Appendix}, we will 
recall facts from semiclassical analysis and present a proof of Egorov's 
theorem for a noisy quantum evolution.
\subsection*{Acknowledgements.} Both authors were partially funded by the 
Agence Nationale de la Recherche, through the project ADYCT (ANR-20-CE40-0017). 
We are both grateful to the Institut de Recherche Math{\'e}matique Avanc{\'e}e 
and to the Laboratoire J. A. Dieudonn\'e, as parts of the paper were written 
during respective visits of the institutes. 
\section{Classical dynamics and Lagrangian states}\label{sec:Class}
In the sequel, $(X,g)$ will be a smooth compact manifold of negative sectional 
curvature without boundary. We denote by $\mathrm{dist}_X$ the geodesic distance 
on $X$. We denote by $T^*X$ the cotangent bundle of $X$, and by 
$\pi_X : T^*X \longrightarrow X$ the canonical projection.  
\par 
By $|\cdot|_x$, we denote the norm on $T^*_xX$ (respectively 
on $T_x X$ whenever convenient) induced by the metric $g$. We equip the 
cotangent bundle $T^*X$ with an arbitrary metric $g_0$ such that the induced geodesic 
distance $\mathrm{dist}_{T^*X}$ on $T^*X$ \ora{satisfies}
$\mathrm{dist}_{T^*X}(\rho_1,\rho_2)\geq c\dist_X(\pi_X (\rho_1),\pi_X (\rho_2))$ 
for some fixed constant $c>0$. This is for instance the case when we 
take $g_0$ to be the Sasaki metric on $T^*X$ induced by $g$.
\subsection{Geodesic flow}
We will denote by $\Phi^t : T^*X\longrightarrow T^*X$ the geodesic flow. 
For each $\lambda>0$,  we shall denote $S_{\lambda}^*X := \{(x,\xi)\in T^*X; 
|\xi|^2=\lambda\}$, and write $\Phi_\lambda^t : S_\lambda^*X \longrightarrow 
S_\lambda^*X$ for the geodesic flow acting on $S_\lambda^*X$. For each 
$\lambda>0$, the geodesic flow $\Phi_\lambda^t$ is Anosov, so for any 
$\rho\in S_\lambda^*X$  we may decompose the tangent spaces 
$T_\rho S_\lambda^*X$ into \emph{unstable}, \emph{neutral} and 
\emph{stable directions}
\begin{equation*}
	T_\rho S_\lambda^*X=E_\rho^+\oplus E_\rho^0 \oplus E_\rho^-.
\end{equation*} 
We refer the reader to \cite[Section 4.2]{IngVog} for the precise definition 
and properties of the spaces $E_\rho^\bullet$, and to \cite{Ebe} for a proof 
of the Anosov property.
\\
\par
%
%
%
We recall that Gronwall's lemma implies that for any $0< \mu_1 < \mu_2$, 
there exists $C_0>0$ such that,
\begin{equation}\label{eq:ExpRate}
\forall \rho, \rho' \in \mathcal{E}_{(\mu_1, \mu_2)}, ~
\forall t\in \R, 
\quad 
 \mathrm{dist}_{T^*X} 
\left( 
	\Phi^t(\rho), 
	\Phi^t(\rho') 
\right) 
\leq C_0 \e^{C_0 |t|}
\mathrm{dist}_{T^*X}(\rho, \rho').
\end{equation}
\subsection{Universal cover}
We will denote the universal cover of $X$ by $\widetilde{\pi}:\widetilde{X}\to X$. 
Since $X$ is a connected Riemannian manifold of negative sectional curvature, 
$\widetilde{X}$ is a simply-connected manifold of negative sectional 
curvature.  We equip $\wt{X}$ and $T^*\wt{X}$ with the lifted Riemannian metrics 
$\wt{g}$ and $\wt{g}_0$, respectively. Equation (\ref{eq:ExpRate}) then holds 
(with the same constant $C_0>0$) when $X$ is replaced by $\wit{X}$.
\par
Note that since the manifold $X$ is compact, there exists $c_0>0$ such that
\begin{equation}\label{eq:DistCover}
\wit{x}, \wit{x}'\in \wit{X} \text{ with } \wit{\pi}(\wit{x}) 
	= \wit{\pi}(\wit{x}')  =x,  \left(\wit{x} \neq \wit{x}'\right) 
	\Longrightarrow \mathrm{dist}_{\wit{X}} (x,x') >c_0.
\end{equation}
\subsection{$C^L$ norms}
In this work,  $X$ is a compact Riemannian manifold without boundary, 
and to define the $C^L$ norms, we fix a finite collection of coordinate 
charts $\{\varphi_\iota\}$ defined in open sets $(\cU_\iota)$ covering $X$, 
with a subordinate partition of unity $1=\sum_\iota\chi_\iota$. We then define 
the $C^L$ norm for functions on $X$ by 
\begin{equation}\label{eq:defCL}
	\| u \|_{C^L(X)} = \max_{\iota \in I}
	\|(\varphi_\iota^{-1})^*\chi_\iota u \|_{C^L(\R^d)}.
\end{equation}
This norm is not intrinsically defined, but taking different coordinate 
patches and cut-off functions in \eqref{eq:defCL} yields an equivalent norm. 
\subsection{Lagrangian states}
A \emph{Lagrangian state} on $X$ is a family of functions $f_h\in C^\infty(X)$ 
indexed by $h\in]0,1]$, of the form
\begin{equation}\label{eq:LagState}
f_h(x)=a(x;h)\e^{i\phi(x)/h},
\end{equation}
where $\phi\in C^\infty(O)$ for some connected and simply connected  open subset 
$O\subset X$ and $a(\cdot;h)\in C^\infty_c(O)$. To a Lagrangian state we can 
associate a \emph{Lagrangian manifold}
\begin{equation*}
	\Lambda_\phi:= \{ (x, d_x \phi) ; x\in O\}\ \subset T^*X.
\end{equation*}
A Lagrangian state is called 
$\lambda$-\emph{monochromatic} if $\Lambda_\phi\subset S^*_\lambda X:= 
\{(x,\xi)\in T^*X; |\xi|_x^2=\lambda\}$, i.e. if 
\begin{equation}\label{eq:LagState_monochrom}
	|d_x \phi|^2 = \lambda \quad  \text{for all } x\in O.
\end{equation}
\begin{definition}
For every $\eta>0$, we say that a Lagrangian manifold 
$\Lambda_{\phi}\subset S_\lambda^*X$ is $\eta$-unstable if, 
for every $\rho\in \Lambda_\phi$ and for every $v\in T_\rho \Lambda_\phi$, 
writing $v= (v_+, v_-, v_0)\in E_\rho^+\oplus E_\rho^-\oplus E_\rho^0$, 
we have 
\begin{equation*}
	|(0,v_-,0)|_{\rho} \leq \eta |v|_\rho.
\end{equation*} 
\end{definition}
\par
Recall that the intrinsic distance $\mathrm{dist}_{\Lambda}(\rho_1, \rho_2)$ 
between two points $\rho_1, \rho_2\in \Lambda$ is the minimal length of 
curves in $\Lambda$ joining $\rho_1$ and $\rho_2$, the length being computed 
using an arbitrary metric on $T^*X$.  We define the 
\emph{distortion} of $\Lambda$ as
\begin{equation}\label{eq:DefDistorsion}
\mathrm{distortion}(\Lambda) 
:= \sup\limits_{\rho_1, \rho_2\in \Lambda} 
	\frac{\mathrm{dist}_{\Lambda}(\rho_1,\rho_2)}{\mathrm{dist}_{T^*X}(\rho_1, \rho_2)}.
\end{equation}
\section{Decomposition of eigenmodes as sums of Lagrangian states}
\label{Sec:DecompoLag}
In the sequel, we will consider linear combinations of Lagrangian states 
that are associated to unstable Lagrangian manifolds which are far away 
from each other.
\begin{definition}
Let $U\subset X$ be an open set, let \ora{ $0<\mu_1 < \mu_2$ be two number which may depend on $h$, but are bounded from above and below independently of $h$}, and 
let $\gamma, \eta, D, \varepsilon>0$. An \emph{$\varepsilon$-sharp, 
$D$-bounded, $h^\gamma$-separated superposition of $\eta$-unstable 
Lagrangian states of energy $(\mu_1,\mu_2)$ in $U$} is an 
$h$-dependent function $g_h$ of the form 
\begin{equation}\label{eqn:TesTropClasseMaxime}
	g_h:= \sum_{j\in J_h} s_j f_j,
\end{equation}
such that
\begin{itemize}
\item  $(s_j)_{j\in J_h}$ is a family of complex numbers indexed by a 
finite ($h$-dependent) set $J_h$ of cardinality
\begin{equation}\label{eq:CardJh}
|J_h| = \mO(h^{-(d-1)\gamma- \varepsilon});
\end{equation}
\item each $f_j$ is a function on $U$ of the form $f_j(x) = a_j(x;h) 
e^{\frac{i}{h} \phi_j(x)}$, where $a_j(\cdot ; h)$ is smooth, supported 
in $U$,  has all its derivatives bounded independently of $h$ and $j$, and 
satisfies $\|a_j (\cdot ; h)\|_{C^0}\leq 1$;
\item there exists $h_0>0$ such that, for every $0<h\leq h_0$, every 
$j,j'\in J_h$ and every $x\in \spt~a_j \cap \spt~a_{j'}$, we have 
$|\partial_x \phi_j(x) - \partial_x \phi_{j'}(x)|> h^{\gamma}$;
\item if we write $\Lambda_j := \{ (x, d_x \phi_j) ; x\in  U\}$, 
then $\Lambda_j$ is $\eta$-unstable, $\lambda_j$-monochromatic for 
some $\lambda_j \in \left(\mu_1,\mu_2\right)$, it has 
distortion $\leq D$,  and we have $\|\phi_j\|_{C^3} \leq D$. Furthermore, all the derivatives of $\phi_j$ are bounded independently of $h$ and $j$.
\end{itemize}
\end{definition}
\par
The following proposition allows us to decompose eigenfunctions of 
$-h^2\Delta_g$ using $\mO(h^{1-d-\varepsilon})$ many Lagrangian states.
\begin{prop}\label{Prop:DecompLag}
Let $0<\mu_1<\mu_2$. There exists $D>0$ such that, for any 
$\varepsilon, \eta>0$,  $0<\gamma<1$, and any $L^2$-normalized 
family of functions $\psi_h$ such that $-h^2\Delta_g \psi_h = E_h \psi_h$ with 
$E_h \in (\mu_1,\mu_2)$, the following holds. We may write
\begin{equation*}
	\psi_h = \sum_{i\in \mathcal{I}_h} g_{i,h} + \mO_{C^\infty}(h^\infty),
\end{equation*}
where each $g_{i,h}$ is an $\varepsilon$-sharp, $D$-bounded,  
$h^\gamma$-separated superposition of $\eta$-unstable Lagrangian 
states of energy $\left(\mu_1 - \ora{h^{1-\varepsilon}}, \mu_2 + 
\ora{h^{1-\varepsilon}} \right)$ in a set $U_i\subset X$ of diameter 
$\leq \varepsilon$, and where $\mathcal{I}_h$ is a finite set of cardinality  
$h^{(d-1)(\gamma-1)- \varepsilon}$.  Additionally, all derivatives of 
the amplitudes and phases of the Lagrangian states are independent of 
$i$, as well as of $j$ and $h$ as in \eqref{eqn:TesTropClasseMaxime}. 
\par
Furthermore, if $\gamma>\frac{1}{2}$, we have
\begin{equation}\label{eq:Orthogonalite}
 \sum_{i\in \mathcal{I}_h} \|g_{i,h}\|_{L^2}^2 \leq \mO(1) \|\psi_h\|_{L^2}^2.
\end{equation}
\end{prop}
Before we present the proof of this result in Section \ref{sec:ProofOfProp} 
below we discuss a local model.
\subsection{A local model}
Before proving Proposition \ref{Prop:DecompLag}, we shall consider a local 
model in $\R^d$. To distinguish functions and operators on manifolds from 
the one in local charts, we will denote the objects in $\R^d$ by bold letters.
\par
In this section, we will use the following notations for projections 
$\R^{2d} \longrightarrow \R^d$: 
\begin{equation*}
	\pi_x(\bx_1,..., \bx_d, \bxi_1,..., \bxi_d) 
= (\bx_1,...,\bx_d), \text{ and }\pi_\xi(\bx_1,..., \bx_d, \bxi_1,..., \bxi_d) 
= (\bxi_1,...,\bxi_d).
\end{equation*}

\ora{In the sequel, we will always write $D_x$ for $\frac{1}{i} \partial_x$.}
\begin{lem}\label{lem:ToyModel}
Let $U\subset \R^{2d}$ be a bounded open set containing $(0,0)$, with 
$\overline{U}\subset (-\sigma \pi, \sigma \pi)^d \times (-\varrho, \varrho)^d$, 
for some $\sigma, \varrho >0$. Let $\bpsi_h$ be a family of $L^2$ 
functions on $\R^d$ with $\|\bpsi_h\|_{L^2}=\mO(1)$\ora{, and let $\mu_h$ be a family of real numbers, bounded independently of $h$}. Suppose that
\begin{equation*}
	\bA (h D_{\bx_1} \ora{+ \mu_h}) \bpsi_h = \mO_{L^2}(h^\infty)
\end{equation*}
whenever $\bA\in \Psi_h^{\comp}(\R^d)$ is microsupported in $U$. 
Let $\bB\in \Psi_h^{\comp}(\R^d)$ be  microsupported in $U$. 
\par
There exists 
$ C>0$ and a sequence  $(a_n)\in \ell^2(\Z^d)$ with $\|a_n\|_{\ell^2}= 
C\| \bB \bpsi_h\|_{L^2}+\mO(h^\infty)$ such that the following holds: 
\par
For any $\varepsilon>0$, we have
\begin{equation*}
	 \bB \bpsi_h(\bx) 
	= \frac{\sigma^{d/2}}{(2\pi)^{d/2}}\sum_{\underset{|n_2|+...+|n_d| < \varrho/(\sigma h)}{n\in \Z^{d} 
		; |n_1 \ora{- h^{-1} \mu_h}| < h^{-\varepsilon}}} a_n e^{ i\sigma n\cdot \bx} 
		+ \mO_{C^\infty((-\sigma \pi, \sigma \pi)^d)}(h^\infty), 
		\quad \bx\in (-\sigma \pi, \sigma \pi)^d.
\end{equation*}
In particular the sum above contains $\mO(h^{-d +1 - \varepsilon})$ 
terms.
\end{lem}
\begin{proof}
\ora{First of all, note that we can always reduce to the case $\mu_h=0$ by multiplying $\psi_h$ by $e^{-\frac{i}{h} \mu_h \bx_1}$. We will thus suppose that $\mu_h=0$ in the rest of the proof.}

By assumption we have that $0\in \pi_x(U) \Subset (- \sigma \pi,  \sigma \pi)^d
=:\cT$.
Let us take $\chi \in C_c^\infty(\cT)$ with 
$\chi \equiv 1$ on $\pi_x(U)$. Let $\bB\in \Psi_h^{comp}(\R^d)$ be 
microsupported in $U$. We then have $ \bB \bpsi_h = \chi \bB \bpsi_h + 
\mO_{\cS(\R^d)}(h^\infty)$.  Furthermore, let $\chi_0\in C_c^\infty(\R^d ; [0+\infty[)$ 
with $\chi_0$ equal to $1$ on $\pi_\xi(U)$, and $\chi_0$ supported inside 
$(-\varrho, \varrho)^d$.  Writing $\chi_1(\bx,\bxi):= \chi_0(\bxi)$, we then 
have thanks to (\ref{eq:MicrolocallyId})
\begin{equation*}
	\bB \bpsi_h = \chi \mathrm{Op}_h^{\mathrm{w}} (\chi_1)  \bB \bpsi_h + 
	\mO_{\cS(\R^d)}(h^\infty).
\end{equation*}
Notice that $\mathrm{Op}_h^{\mathrm{w}} (\chi_1)$ is selfadjoint since we work with the 
Weyl quantization.
\par
Now, since  $\chi \mathrm{Op}_h^{\mathrm{w}} (\chi_1)  \bB \bpsi_h$ is a smooth compactly 
supported function in $\cT$, we may expand it in a 
Fourier series. Write $e_n(\bx) = (\sigma/2\pi)^{d/2}e^{i \sigma n\cdot \bx}$, 
$n\in\Z^d$, and 
$a_n := \langle  \chi \mathrm{Op}_h^{\mathrm{w}} (\chi_1) \bB \bpsi_h, e_n \rangle = 
\langle  \chi  \bB \bpsi_h, e_n \rangle + \mO(h^\infty)$.  Here, 
$\langle \cdot , \cdot \rangle$ denotes the standard sesquilinear 
$L^2$ scalar product. Then, we have in $\cT$
\begin{align*}
\bB \bpsi_h &=  \chi \mathrm{Op}_h^{\mathrm{w}} (\chi_1)  \bB \bpsi_h + \mO_{C^\infty(\cT)}(h^\infty)\\
&= \sum_{n\in \Z^d} a_n e_n + \mO_{C^\infty(\cT)}(h^\infty)\\
&=  \sum_{n\in \Z^d} \langle \bB \bpsi_h, \mathrm{Op}_h^{\mathrm{w}} (\chi_1) 
	\chi  e_n \rangle e_n + \mO_{C^\infty(\cT)}(h^\infty).
\end{align*}
Note that, by Parceval's identity, we have $\|a_n\|_{\ell^2} = C_1 
\|\chi \mathrm{Op}_h^{\mathrm{w}} (\chi_1) \bB_h \bpsi_h\|_{L^2} = \|\bB_h \bpsi_h\|_{L^2} 
+ \mO(h^\infty)$ for some $C_1>0$.  An application of the second part of the 
non-stationary phase Lemma \ref{Lem:NonStat} implies that
%
%
\begin{equation*}
	\sum_{n\in \Z^d ; |n|\geq \varrho/(\sigma h)} 
		\left| \langle  \bB \bpsi_h, \mathrm{Op}_h^{\mathrm{w}} (\chi_1) \chi  e_n \rangle \right| 
		= \mO(h^\infty).
\end{equation*}
Therefore, we have that in $\cT$
%
\begin{equation*}
\bB \bpsi_h = \sum_{n\in \Z^d ; |n| < \varrho/(\sigma h)} a_n e_n 
	+ \mO_{C^\infty(\cT)}(h^\infty).
\end{equation*}
By integration by parts, we find that 
\begin{align*}
\langle\chi  \bB  \bpsi_h, e_n \rangle 
	= \frac{1}{(h n_1)^p} \langle  (h \ora{D_{\bx_1}})^p \chi \bB \bpsi_h, e_n \rangle.
\end{align*}
Using a commutator argument, we see that for any pseudo-differential operator $\bB^1\in \Psi_h^{comp}(\R^d)$ 
microsupported inside $U$, we may find a pseudo-differential operator  
$\bB^2\in \Psi_h^{\comp}(\R^d)$ microsupported inside $U$ such that
$(h \ora{D_{\bx_1}})  \bB^1 =  \bB^1 (h \ora{D_{\bx_1}}) + h \bB^2$. 
We may thus prove inductively that there exist pseudodifferential operators 
$\bB^{(j,p)}\in \Psi_h^{\comp}(\R^{2d})$ microsupported in $U$ such that 
\begin{equation*}
	(h \ora{D_{\bx_1}}) ^p \chi\mathrm{Op}_h^{\mathrm{w}} (\chi_1) \bB 
	= \sum_{j=0}^p h^j \bB^{(j,p)} (h \ora{D_{\bx_1}})^{p-j}.
\end{equation*}
When applying $(h D_{\bx_1})^p \chi\mathrm{Op}_h^{\mathrm{w}} (\chi_1)  \bB$ to $\bpsi_h$, all the 
terms with $j\neq p$ are $\mO(h^\infty)$ by assumption, so that we have 
\begin{equation*}
	\|(h D_{\bx_1})^p \chi\mathrm{Op}_h^{\mathrm{w}} (\chi_1)\bB \bpsi_h\|_{L^2}= \mO(h^p).
\end{equation*}
Therefore, if we have $|h n_1| > h^{1-\varepsilon}$ for some $\varepsilon>0$, 
we have 
\begin{equation*}
	\langle \chi\mathrm{Op}_h^{\mathrm{w}} (\chi_1)\bB \bpsi_h, e_n \rangle  = \mO(h^\infty).
\end{equation*}
The result follows.
\end{proof}
\subsection{Proof of Proposition \ref{Prop:DecompLag}}
\label{sec:ProofOfProp}
This section is devoted to the proof of Proposition \ref{Prop:DecompLag}.
\\
\par
Let $0<\mu_1<\mu_2$, and let $E_h\in (\mu_1,\mu_2)$.  \ora{We will also write $\mu_{12}= \frac{\mu_1+\mu_2}{2}$.}
\\
\\
\textbf{Step 1: Classical normal form.}
The operator $P_h = -h^2\Delta_g \ora{- \mu_{12}}$ has principal symbol $p(x,\xi)= |\xi|^2-\ora{\mu_{12}}$. 
In this step, we shall use Darboux's theorem to write $p$ in a simpler way 
in a chart. A point in $\R^{2d}$ will be denoted by $(\boldsymbol{x}_1,\dots
\bx_d, \bxi_1, ... \bxi_d)$, and the canonical basis of $\R^{2d}$ will be 
denoted by $\left(e_1,..., e_{2d} \right)$.
\par
Let $\rho_0\in S^*X\subset T^*X$. Let $\cU_\iota\subset X$ be an open set to which 
$\pi(\rho_0)$ belongs, as in the discussion before (\ref{eq:defCL}). Up to 
adding a constant vector to $\varphi_\iota$, we may suppose that 
$\varphi_\iota (\pi(\rho_0)) = 0$. We then lift $\varphi_\iota$ to a symplectic 
map $\kappa_\iota :  T^*X \supset B(\rho_0, \varepsilon_0) \longrightarrow 
V_{\rho_0}\subset \R^{2d}$ for some $\varepsilon_0>0$, with $\kappa_{\iota}(x,\xi)=
(\varphi_\iota(x),\ora{\left((d_x\varphi_\iota)^{-1}\right)^T}\xi)$ and 
$\kappa_{\iota}(\rho_0) = (0,0)$. Here, $B(\rho_0, \varepsilon_0)$ denotes the 
geodesic ball of radius $\varepsilon_0$ and centered at $\rho_0$. 
\par
Up to shrinking $\varepsilon_0$, we may find a symplectomorphism 
$\kappa'_{\rho_0}: V_{\rho_0} \longrightarrow \wt{V}_{\rho_0}\subset
\R^{2d}$ such that 
\begin{equation}\label{eq:ConjugClassic}
p \circ (\kappa'_{\rho_0} \circ \kappa_{\iota})^{-1} = \bxi_1
\end{equation} 
and $\kappa'_{\rho_0}(0,0)= (0,0)$. This follows from the variant of 
Darboux's theorem presented in \cite[Theorem 21.1.6]{Ho84} (with $k=1, j=0$). 
Since the Hamilton vector field $H_{p_h}$ becomes $H_{\bxi_1}$ in those 
new coordinates, we see that 
up to composing $\kappa'_{\rho_0}$ with a linear map, we may suppose that 
$d_{\rho_0}(\kappa'_{\rho_0} \circ \kappa_{\iota}) (E^-_{\rho_0}) 
= \mathrm{span} (e_2,... ,e_d)$.
\par
The map $\kappa'_{\rho_0}$ is built by solving a differential equation whose coefficients depend smoothly on the metric $g$ in a neighborhood of $\rho_0$.
In particular, 
by compactness of $X$, we may find for any $L\in \N$ a constant $C_L>0$ which 
does not depend on $\rho_0$ such that
\begin{equation}\label{eq:CLKappa}
 \|\kappa'_{\rho_0}\|_{C^L(B(\rho_0, \varepsilon_0))}\leq C_L.
\end{equation}
Now, since $\rho \mapsto E^-_\rho$ is (Hölder) continuous, for any $\rho$ 
close enough to $\rho_0$, the affine spaces $d_{\rho} (\kappa_{\rho_0}' 
\circ \kappa_{\iota}) (E^-_{\rho})$ and $\kappa_{\rho_0}' \circ \kappa_{\iota}(\rho) + \mathrm{span} 
(e_2,... ,e_d)$ will make a small angle at $\rho$. This implies that, for 
any $\varepsilon_1>0$ there exists an $\varepsilon'_0>0$ such that, 
if $|\bxi| < \varepsilon'_0$,  the manifold $(\kappa'_{\rho_0} 
\circ \kappa_{\iota})^{-1}(B(0,\varepsilon_0) \times \{\bxi\})$ is 
$\varepsilon_1$-unstable. Here, $B(0, \varepsilon_0)$ denotes the open 
ball in $\R^d$ of radius $\varepsilon_0$ and centered at $0$. Furthermore, 
thanks to \eqref{eq:ConjugClassic}, this manifold is $(\bxi_1 + \ora{\mu_{12}})$-monochromatic.
\\
\\
\textbf{Step 2: Quantum normal form.} We now quantize the map 
$(\kappa_{\rho_0}' \circ \kappa_{\iota})^{-1}$, as in Appendix \ref{sec:FIO}. 
First notice that $\kappa_{\rho_0}'$ satisfies (\ref{eq:Block}), 
and thus, that $\kappa_{\rho_0}'$ can be quantized as in (\ref{eq:LocalFIO}), 
since
the manifold $E_{\rho_0}^-$ is transverse 
to the vertical fibers for any $\rho_0$. 
Hence, by the discussion in Appendix \ref{sec:FIO} and by \eqref{eq:LocalFIO3b} 
in particular, we can find an operator $S_{\rho_0, h} : L^2( \R^d) \to L^2(X)$, 
having a microlocal inverse $T_{\rho_0, h}$, such that, writing 
$U= B(\rho_0, \varepsilon_{\rho_0})$, $\kappa =\kappa'_{\rho_0} 
\circ \kappa_{\iota}$ and $U' = \kappa (U)$, for all 
$a'\in S^\comp(U')$
\begin{equation}\label{eq:FIO1}
\begin{split}
&S_{\rho_0,h} \Op_h(a') T_{\rho_0, h}
	= \Op_h(a) + \mO(h^\infty)_{\Psi^{-\infty}},\\
& \Op_h(a') 
	= T_{\rho_0, h}\Op_h(a)S_{\rho_0,h} + \mO(h^\infty)_{\Psi^{-\infty}},\\
\end{split}
\end{equation}
where $a\in S^\comp(X)$ with $a = a' \circ \kappa + \mO(h)$ near 
$\rho_0$ and $\WF_h(S_{\rho_0,h} \Op_h(a') T_{\rho_0, h}) 
\subset \kappa^{-1}(\WF_h(\Op_h^w(a)))$. Moreover, 
\begin{equation}\label{eq:FIO1b}
	T_{\rho_0, h}= S_{\rho_0,h} ^{-1} \quad \text{ microlocally near }
	U\times U.
\end{equation}
%
%
%
Arguing as in the proof of \cite[Theorem 12.3]{Zw12} or 
\cite[Proposition 26.1.3]{Ho85}, we may compose  
$S_{\rho_0, h}$ with an elliptic pseudo-differential operator 
$A_{\rho_0}\in \Psi_h^0(\R^d)$, and $T_{\rho_0,h}$ with the inverse 
$A_{\rho_0}^{-1}$, to obtain new operators $\wit{S}_{\rho_0, h} = 
S_{\rho_0, h}\circ A_{\rho_0}$ and $\wit{T}_{\rho_0, h}=A_{\rho_0}^{-1}
\circ T_{\rho_0,h}$ such that 
%
\begin{equation}\label{eq:FIO2}
	P_h 
	= \wit{S}_{\rho_0,h} (h D_{x_1}) \wit{T}_{\rho_0,h} 
	\quad \text{ microlocally near }
	U\times U.
\end{equation}
Note that operators $\wit{S}_{\rho_0, h}$ and $\wit{T}_{\rho_0, h}$ 
also satisfy \eqref{eq:FIO1}, \eqref{eq:FIO1b}. 
%
\\
\\
\textbf{Step 3: Partition of unity.} 
Let $\varepsilon, \varepsilon_1>0$. For every $\rho\in 
\mathcal{E}_{(\mu_1, \mu_2)}$, the first two steps allow us to find 
$\varepsilon_\rho \in (0, \varepsilon)$, $\zeta_\rho>1$ and a symplectomorphism 
$\kappa_\rho$ from $B(\rho,  \zeta_\rho \varepsilon_\rho)$ to a neighborhood 
$\boldsymbol{U}_\rho$ of $(0,0)\in \R^{2d}$ such that 
\begin{itemize}
\item $\kappa_\rho^*\,\ora{p} = \boldsymbol{\xi}_1$;
\item There exist $\sigma_\rho,  \varrho_\rho>0$ such that 
	$ \bU_\rho':= \kappa_\rho(B(\rho, \varepsilon_\rho)) \Subset  
	(-\sigma_\rho \pi, \sigma_\rho \pi)^d \times (-\varrho_\rho, \varrho_\rho)^d 
	\subset \boldsymbol{U}_\rho$.
\item Up to possibly shrinking $\varepsilon_\rho$, we may reduce $\sigma_\rho$ 
and $\varrho_\rho$ so that, for any $\bxi\in (-\varrho_\rho, \varrho_\rho)^d$ 
with $|\bxi|< \varrho_\rho$, the manifold $\Lambda_{\rho, \bxi} 
:= \kappa_\rho^{-1}((-\varrho_\rho, \varrho_\rho)^d \times \{\bxi\}))$ is 
$\varepsilon_1$-unstable;
\item There exist operators $T_\rho : L^2(X)\longrightarrow L^2(\R^d)$ and 
$S_\rho : L^2(\R^d) \longrightarrow L^2(X)$ such that \eqref{eq:FIO1}, 
\eqref{eq:FIO1b}, \eqref{eq:FIO2} hold with 
$U= B(\rho, \zeta_\rho \varepsilon_\rho)$, $U'= \kappa_\rho(U)$.
\end{itemize}
If $\varepsilon_1$ has been chosen small enough, the second point above 
implies that $\Lambda_{\rho,\bxi}$ is transverse to the vertical fibers, 
and thus projects smoothly onto $X$. Since it is relatively compact and 
simply connected, it may be written as 
\begin{equation*}
\Lambda_{\rho,\bxi} = \{(x, d_x \phi_{\rho,\bxi}), x\in V_{\rho,\bxi}\}
\end{equation*}
for some smooth function $\phi_{\rho,\bxi}$ defined on some open set 
$V_{\rho,\bxi} \subset \pi_X \left(B(\rho, \zeta_\rho \varepsilon_\rho)\right)$. 
The manifold $\Lambda_{\rho, \bxi}$ is $(\ora{\mu_{12}}+ \xi_1)$-monochromatic, and
is thus $\lambda$-monochromatic for some $\lambda\in (\mu_1- \ora{h^{1-\varepsilon}}, 
\mu_2+ \ora{h^{1-\varepsilon}})$ provided $\varrho_\rho$ has been chosen small enough.
Furthermore, thanks to (\ref{eq:CLKappa}), for every $L\in \N$, there exists 
$C'_L>0$ which does not depend on $k$, $\varepsilon$ or $\varepsilon_1$ such 
that 
\begin{equation*}
\|\phi_{\rho, \bxi}\|_{C^L} \leq C'_L.
\end{equation*}
By compactness, we may find a finite set of points $(\rho_k)_{k\in \mathcal{K}}$ such 
that $\mathcal{E}_{(\mu_1, \mu_2)} \subset  \bigcup_{k\in \mathcal{K}} 
B(\rho_k, \varepsilon_k)$.  In the rest of the proof, we will write all the 
quantities defined above with indices $k$, instead of indices $\rho_k$, to 
lighten notations.
\\
\par
Let us consider three families of functions on $T^*X$: $a_k \in C_c^\infty
(B(\rho_k,  \varepsilon_k); [0,1])$, $a'_k, a_k''\in 
C_c^\infty(B(\rho_k, \zeta_k \varepsilon_k); [0,1])$ ($k\in \mathcal{K}$), 
with $\sum_{k\in \mathcal{K}} a_k \equiv 1$ on $S^*X$, $a'_k\equiv 1$ on 
$B(\rho_k, \varepsilon_k)$ and $a''_k \equiv 1$ on the support of $a'_k$. 
We quantize these into $A_{k} :=\Op_h(a_k)$, $A'_{k} :=\Op_h(a'_k)$, 
$A''_{k} :=\Op_h(a''_k)$, so that we have
\begin{equation}\label{eq:PartitionUnity}
\sum_{k\in \cK} A_k = \mathrm{Id}  \text{ microlocally near }
	 \mathcal{E}_{(\mu_1, \mu_2)}\times \mathcal{E}_{(\mu_1, \mu_2)}.
\end{equation}
\\
\\
\textbf{Step 4: Applying the toy model.}
First of all, note that for every $k\in \cK$,
\begin{equation}\label{eq:Commut}
A'_k P_h A''_k  = A'_k P_h + \mO(h^\infty)_{\Psi^{-\infty}},
\end{equation}
where $ \mO(h^\infty)_{\Psi^{-\infty}}$ denotes an operator of residual 
class $h^\infty \Psi^{-\infty}$. 
Let $\psi_h$ be such that $-h^2\Delta_g \psi_h = E_h \psi_h$, 
with $\|\psi_h\|_{L^2}=1$. We then have, thanks to (\ref{eq:FIO2}) and to 
(\ref{eq:Commut})
\begin{align*}
\left(T_k A'_k S_k \right)  (h D_{x_1}  \ora{+ \mu_{12} - E_h})   \left( T_k A''_k \psi_h\right) 
	&= T_k \left( A'_k  (P_h  \ora{+ \mu_{12} - E_h})  A''_k \right) \psi_h + \mO_{C^\infty}(h^\infty)\\
	&= T_k A'_k (P_h \ora{+ \mu_{12} - E_h})   \psi_h + \mO_{C^\infty}(h^\infty)\\
	&= \mO_{C^\infty}(h^\infty).
\end{align*}
\par
Note that by \eqref{eq:CoordChange}, the error term has compact support. 
\par
By \eqref{eq:FIO1}, we know that  
$T_k A'_k S_k = \Op_h(\boldsymbol{a}'_k)+ \mO(h^\infty)_{\Psi^{-\infty}}$, where $\boldsymbol{a}'_k>0$ 
in $\bU'_k$. Therefore, $T_k A'_k S_k$ is microlocally invertible near $\bU'_k \times \bU'_k$. 
 Hence, whenever $\bA\in \Psi_h^{\comp}(\R^d)$ is microsupported in 
$\bU'_k$, we have $\bA (hD_{x_1} \ora{+ \mu_{12} - E_h} ) \left( T_k A''_k \psi_h\right) = \mO_{L^2}(h^\infty)$. 
We may apply Lemma \ref{lem:ToyModel} to $\bB = T_k A_k S_k$, which is 
microsupported in $\bU'_k$ thanks to \eqref{eq:FIO1}. 
Let $\bchi_k \in C_c^\infty((- \sigma_k\pi,  \sigma_k\pi)^d)$ be equal to $1$ 
on $\pi (\bU_k')$. We deduce that in $(- \sigma_k\pi,  \sigma_k\pi)^d$
\begin{align*}
T_k A_k  \psi_h&= \left(T_k A_k S_k \right) T_k A''_k \psi_h 
	 + \mO_{C^\infty((- \sigma_k\pi,  \sigma_k\pi)^d)}(h^\infty)\\
&= \bchi_k \left(T_k A_k S_k \right) T_k A''_k \psi_h 
	+ \mO_{C^\infty((- \sigma_k\pi,  \sigma_k\pi)^d)}(h^\infty)\\
&= \bchi_k  \frac{\sigma^{d/2}}{(2\pi)^{d/2}} \sum_{\underset{|n| < \varrho_k/(\sigma_k h) }{n\in \Z^{d} ; 
|n_1 \ora{+ h^{-1}(E_h - \mu_{12}})| < h^{-\varepsilon}}} a^k_n e^{i\sigma_k n\cdot x} 
	+ \mO_{C^\infty((- \sigma_k\pi,  \sigma_k\pi)^d)}(h^\infty),
\end{align*}
for some sequence $(a^k_n)_{n\in \Z^d}$ with 
\begin{equation}\label{eq:Meeeec}
	\|a^k\|_{\ell^2} = 
C_k \| T_k A_k'' \psi_h  \|_{L^2} +\mO(h^\infty).
\end{equation}
\textbf{Step 5: Applying Fourier Integral Operators to Lagrangian states.}

Since $-h^2\Delta_g\psi_h = E_h\psi_h$, it follows that $\WF_h(\psi_h)
\subset \mathcal{E}_{\ora{[}\mu_1,\mu_2\ora{]}}$, see for instance 
\cite[Proposition E.39]{DyZw19}. Thanks to (\ref{eq:PartitionUnity}), we 
may write
\begin{equation}\label{eq:Martin,EcrisTonHDRPlutotQueDeRegarderLeNomDesEquations}
\begin{aligned}
\psi_h &= \sum_{k\in \cK} A_k \psi_h +\mO_{C^\infty(X)}(h^\infty)\\
&= \sum_{k\in \cK} S_k \left[\sum_{n\in \Z^{d} ; |n_1\ora{+ h^{-1}(E_h - \mu_{12}})| < h^{-\varepsilon}} 
	\mathbf{1}_{\{|n| < \varrho_k/(\sigma_k h)\} }a^k_n \be_{k,n}  \right] 
	+ \mO_{C^\infty(X)}(h^\infty),
\end{aligned}
\end{equation}
where $\be_{k,n}(\bx) = \bchi_k(\bx)  \frac{\sigma_k^{d/2}}{(2\pi)^{d/2}}e^{i\sigma_k n\cdot \bx} $.
\\
\par
Next, recall from the second step that $S_k$ may be written as 
$S_{k, \varphi_k} \circ S_k' \circ B_k$, with $S_k'$ as in (\ref{eq:LocalFIO}), 
and $S_{k, \varphi_k}$ as in (\ref{eq:CoordChange}), and where $B_k$ is a 
$\Psi$DO that is elliptic at $(0,0)$. Applying Lemma \ref{lem:FIOonLag} twice, 
we obtain that  $(S_k' \circ B_k) \be_{k,n}$ is a Lagrangian state, of the form 
\begin{equation*}
	\bchi_{n,k} (\bx ; h) e^{\frac{i}{h} \phi_{n,k}(\bx)} 
	+ \mO_{\cS(\R^d)}(h^\infty),
\end{equation*}
where $\{(\bx,  \partial_{\bx}  \phi_{n,k}(\bx))\} = (\kappa_k')^{-1} 
\left(\{ (\bx,  h \sigma n) ; \bx\in  (-\sigma_k \pi, \sigma_k \pi)^d \} 
\right)$, and where for any $\ell\in\N$,  $\|\bchi_{n,k} \|_{C^\ell}$ is 
bounded independently of $h$, $k$ and $n$. We thus have 
\begin{equation}\label{eq:MorceauDecompoLag}
(S_k\be_{k,n})(x) 
	= \chi_{k,n}(x; h) e^{\frac{i}{h} \phi_{k,n}(x)} + 
	\mO_{C^\infty(X)}(h^\infty),
\end{equation}
with
\begin{equation*}
	\Lambda_{k,n} := \{ (x, \partial_x \phi_{k,n})\} 
	=\kappa_k^{-1} \left(\{ (\bx,  h \sigma n) ; \bx\in 
		(-\sigma_k \pi, \sigma_k \pi)^d \} \right)
\end{equation*}
and with $\|\chi_{k,n}\|_{C^\ell}$ bounded independently of $h$, $k$ and 
$|n| < \varrho_k/(\sigma_k h)$. Furthermore, $\chi_{k,n}$ may be 
written as $\chi_{k,n} (x;h) = \sum_{j=0}^J h^j \chi_{j,k,n}(x) +  
h^{J+1} R_{J+1}(x;h)$, with the  $\chi_{j,k,n}$ and $R_{J+1}$ having 
their $C^\ell$ norms and support bounded independently of $h$, $k$ and 
$|n| < \varrho_k/(\sigma_k h)$. Finally,  since $\kappa_k^{-1}$ is 
smooth,  the derivatives of $\phi_{k,n}$ are bounded independently of $k,n$, and
there exists $C_1, C_2>0$, independent of $h$ such that, for any 
$n,n'$ with $|n|, |n'| < \frac{\varrho_k}{\sigma_k h}$,
\begin{equation}\label{eq:EcartsLag}
C_1 h |n'-n|
	\leq \mathrm{dist} \left(\Lambda_{k,n},\Lambda_{k,n'} \right) 
	\leq C_2 h |n'-n|.
\end{equation}
\\
\textbf{Step 6: Regrouping the Lagrangian states.}
Let us set $N_h := \lfloor h^{\gamma - 1 - \varepsilon} \rfloor$,  for 
$\varepsilon>0$ small enough, so that 
\begin{equation*}
	\{n\in \Z^d;  |n_1\ora{+ h^{-1}(E_h - \mu_{12})}| < h^{-\varepsilon} \} 
	=  \bigsqcup_{\underset{|n_1| < h^{-\varepsilon}}{n_1\in \Z}} 
	\bigsqcup_{m\in (\Z / N_h \Z)^{d-1}} 
		\left[ (n_1,m)+ \{0\}\times 
		(N_h \Z )^{d-1}\right].
\end{equation*}
Write 
\begin{equation*}
	\cI'_h:= \{ n_1\in \Z ; |n_1\ora{+ h^{-1}(E_h - \mu_{12})}| < h^{-\varepsilon}\} \times  (\Z / N_h \Z)^{d-1},
\end{equation*}
and, for every $(n_1,m)\in \cI'_h$, 
\begin{equation*}
	\mathcal{J}_{n_1,m} 
	:= \{ n=(n_1,..., n_d)\in (n_1,m)+ \{0\}\times (N_h \Z)^{d-1} 
		\text{ such that } |n| < C_0 h^{-1} \}.
\end{equation*}
By to \eqref{eq:Martin,EcrisTonHDRPlutotQueDeRegarderLeNomDesEquations} 
and (\ref{eq:MorceauDecompoLag}), we have
\begin{align*}
\psi_h = \sum_{k\in \cK} \sum_{(n_1,m) \in \cI_h'} \sum_{n\in \mathcal{J}_{n_1,m}} 
	a_n^k \chi_{k,n}(\cdot ; h) e^{\frac{i}{h} \phi_{k,n}} 
	+ \mO_{C^\infty}(h^\infty).
\end{align*}
For each $k\in \cK$ and $(n_1,m)\in \cI_h'$, the function 
\begin{equation}\label{eq:Meeeec2}
	g_{k,n_1,m}(x) 
	:= \sum_{n\in \mathcal{J}_{n_1,m}} a_n^k \chi_{k,n}(x; h) 
	e^{\frac{i}{h} \phi_{k,n}(x)}
\end{equation}
is a $C\varepsilon$-sharp (for some $C>0$ 
independent of $h$, $k$ and $(n_1,m)$),  $h^\gamma$-separated \ora{superposition of} $\eta$-unstable 
Lagrangian state of energy $(\mu_1 - \ora{h^{1-\varepsilon}}, \mu_2+ \ora{h^{1-\varepsilon}})$
in $U=  \pi_X(B(\rho_k,  \varepsilon_k)) \subset\pi_X(B(\rho_k, \varepsilon))$, 
where the $h^\gamma$-separatedness comes from (\ref{eq:EcartsLag}).
%
\\
\par
Thanks to (\ref{eq:CLKappa}), we have $\|\phi_{k,n}\|_{C^3} \leq D$ for some 
$D$ independent of $h$, $k$ and $(n_1,m)$. Since the manifold 
$\{ (\bx,  h \sigma n) ; \bx\in  (-\sigma_k \pi, \sigma_k \pi)^d \}\subset 
\R^{2d}$ has distortion $1$ (here we used tacitly the standard Euclidean 
metric on $\R^{2d}$ and on $\R^d$),  its image $\Lambda_{k,n}$  by $\kappa^{-1}$ 
is $D'$-distorted for some $D'>0$ independent of $h$, $k$ and $(n_1,m)$.
Noting that $\cI_h = \cK \times \cI_h'$ has cardinal 
$\mO(h^{(d-1)(\gamma - 1)-C\varepsilon})$, the first part of Proposition 
\ref{Prop:DecompLag} follows, by taking $\varepsilon$ possibly smaller.
\\
\par
Finally, we note that, if $\gamma> \frac{1}{2}$, the $h^\gamma$-separatedness, 
along with Lemma \ref{Lem:NonStat} implies that 
\begin{align*}
\|g_{k,n_1,m}\|_{L^2}^2 
	&= \sum_{n\in \mathcal{J}_{n_1,m}} |a_n^k|^2 \|\chi_{k,n}(x; h)\|_{L^2}^2 
		+ \mO(h^\infty)\\
&\leq C  \sum_{n\in \mathcal{J}_{n_1,m}} |a_n^k|^2 + \mO(h^\infty).
\end{align*}

In view of \eqref{eq:Meeeec} and the fact that $\cK$ is finite, 
equation \eqref{eq:Orthogonalite} follows, thus 
concluding the proof of Proposition \ref{Prop:DecompLag}.
\section{Propagating superpositions of Lagrangian states}\label{sec:Proof}
Let $0<\varepsilon < \mu_1 \ora{\leq }\mu_2$
, and let $g_{h}$ be an 
$\varepsilon$-sharp, $D$-bounded, $h^{\beta - \varepsilon}$-separated 
superposition of $\eta$-unstable Lagrangian states of energy 
$(\mu_1- \ora{h^{1-\varepsilon}},\mu_2+\ora{h^{1-\varepsilon}})$ in a set $U$, as in 
Proposition \ref{Prop:DecompLag}:
\begin{equation}\label{eq:Meeec3}
	g_h(x) = \sum_{j\in J_h} s_j a_j(x;h) e^{\frac{i}{h}\phi_j(x)}.
\end{equation}
We shall write 
\begin{equation*}
	g_{h,t,\omega}:= e^{i \frac{t}{h} P_h^\delta} g_h,
\end{equation*}
and we shall be interested in $\|g_{h,t,\omega}\|_{L^\infty}$.
\\
\par
The random function $g_{h,t,\omega}$ may be described using the results 
of \cite{IngVog}. To apply the results from this paper, we must 
assume that the Lagrangian manifolds $\Lambda_j$ have distortion $\leq D$, 
with $ \|\phi_j\|_{C^3}\leq D$, and are $\eta$-unstable, for some 
$\eta(D)>0$. Indeed, this assumption is justified by Proposition 
\ref{Prop:DecompLag}. Furthermore, the time $t>0$ being fixed (and 
independent of $h$), we will assume that the set $U$ has diameter 
$< \frac{c_0}{2 C_0 e^{C_0 t}}$, with $c_0$ as in (\ref{eq:DistCover}) 
and $C_0$ as in (\ref{eq:ExpRate}). 

Note that, in \cite{IngVog},  all the estimates on the propagation of a Lagrangian state $a(\cdot; h) e^{ \frac{i}{h} \phi}$ associated with a Lagrangian manifold $\Lambda$ depend only on the distortion and instability of $\Lambda$,  and on bounds on the derivatives of $a(\cdot ; h )$ and $\phi$. Therefore,  when propagating a function of the form (\ref{eq:Meeec3}), all the estimates are uniform with respect to the parameter $j\in J_h$. Furthermore, if, as in Proposition \ref{Prop:DecompLag}, the function $g_h$ depends on an index $i\in \cI_h$, then all the estimates will be independent of $i$.
\subsection{Propagation of Lagrangian states}
Given \eqref{eq:Meeec3}, let us write 
\begin{equation*}
	\Lambda_j := \{ (y, d_y \phi_j) ; y\in U \}.
\end{equation*}
Strictly speaking, the WKB method presented in \cite{IngVog} applies only 
to Lagrangian states $a_j(x) e^{\frac{i}{h}\phi_j(x)}$ with 
$|\partial_x \phi_j|=1$ on $\spt~a_j$.  However, if 
$|\partial_x \phi_j|^{2}=\lambda_j$ for some $\lambda_j\in 
\left(\mu_1 - h^{1-\varepsilon}, \mu_2 + 
h^{1-\varepsilon} \right)$,
we may just rescale the metric on the manifold $X$ by a factor $\lambda_j$, 
and apply the results of \cite[Sections 5,6,7]{IngVog}.
%
\\
\par
First of all, \cite[Equation (7.25)]{IngVog} implies that, for any $N\in \N$, 
we may write
\begin{equation}\label{eq:Rappel}
g_{h,t,\omega}(x)
	= \sum_{j\in J_h} s_j \sum_{\substack{\pi(\wit{x})=x \\ \wit{x}\in \wit{X}}} 
	b_{j,N}(t, \wit{x} ; \delta,h) e^{\frac{i}{h}\phi_{j,t, \delta} (\wit{x})} 
	+ \mO_{C^0} (h^N).
\end{equation}

Let $\wit{\Lambda}_j$ be a fixed lift of $\Lambda_j$ to the universal cover 
$\wt{X}$ of $X$. 
Then the quantity $ b_{j,N}(t, \wit{x} ; \delta,h)$ vanishes, unless there 
exists a $\xi_j(\wit{x}) \in T_{\wit{x}}^*\wit{X}$ such that 
$\wit{\Phi}^{-t}_\delta (\wit{x},\xi_j(\wit{x}))\in \wit{\Lambda}_j$. 
Recall here \eqref{eq:HamiltonianFlow}. Thanks to the 
assumptions we made on the diameter of $U$ at the beginning of the section, 
we see from (\ref{eq:ExpRate}) that for each $j$ and each $x\in X$, there 
exists at most one $\wit{x}$ such that $ b_{j,N}(t, \wit{x} ; \delta,h)$ 
is non-zero.
\par
For each $x\in X$, let us denote by $J_h(x)$ the set of indices $j\in J_h$ such 
that such a $\wit{x}$ exists. For each $j\in J_h(x)$, we then denote 
by $\wit{x}_{j,t}$ the corresponding lift of $x$, and we shall write 
$b_{j,N}(t,x ; \delta,h)$ and $\phi_{j,t, \delta} (x)$ instead of 
$b_{j,N}(t,\wit{x}_{j,t} ; \delta,h)$ and $\phi_{j,t, \delta} (\wit{x}_{j,t})$. 
The functions $x\mapsto  b_{j,N}(t,  x ; \delta,h)$ and 
$x\mapsto \phi_{j,t, \delta} (x)$ are smooth functions, since 
$y\mapsto \wit{y}_{j,t}$ is a smooth function in a neighborhood of $x$. 
We may thus write
\begin{equation}\label{eq:Rappel2}
g_{h,t,\omega}(x)
	= \sum_{j\in J_h(x)} s_j b_{j,N}(t, x ; \delta,h) e^{\frac{i}{h}\phi_{j,t, \delta} (x)} 
	+ \mO_{C^0} (h^N).
\end{equation}
Setting $b_{j,N}(t, x ; \delta,h) = 0$ if $j\notin J_h(x)$, we may also write
\begin{equation}\label{eq:Rappel3}
g_{h,t,\omega}(x)
	= \sum_{j\in J_h} s_j b_{j,N}(t, x ; \delta,h) 
	e^{\frac{i}{h}\phi_{j,t, \delta} (x)} + \mO_{C^0} (h^N).
\end{equation}
Let us now explain more precisely what the phases 
$\phi_{j,t, \delta} (x)$ and the amplitudes $b_{j,N}(t, x ; \delta,h)$ 
appearing in (\ref{eq:Rappel3}) are. 
\subsection*{The amplitude} We may write
\begin{equation}\label{eq:Rappel3a}
	b_{j,N} = \sum_{k=0}^{K_N} h^k b_{j,k},
\end{equation}
and \cite[Proposition 6.5]{IngVog} implies that 
\begin{equation}\label{eq:EstimDerivAmpl}
\|b_{j,k} (t,\cdot ; \delta, h) \|_{C^{L}} = \mO(h^{(-2k+L)(\beta+ \varepsilon)})
\end{equation}
for any $\varepsilon>0$ small enough. In particular, we have 
$\|h^k b_{j,k} (t,\cdot ; \delta, h) \|_{C^0} \leq C_k h^{k\Gamma}$ since 
$\Gamma < 1-2\beta$, see \eqref{eq:CondGamma.b}, so that 
\begin{equation*}
	b_{j,N} = b_{j,0} + \mO(h^\Gamma).
\end{equation*}
Now, thanks to \cite[Proposition 6.7]{IngVog}, we have
\begin{equation}\label{eq:ApproxAmpl}
b_{j,0} (t, \cdot ; \delta, h) 
	= \hat{b}_{j,0} (t, \cdot) + \mO(h^{\Gamma}),
\end{equation}
where $\hat{b}_{j,0}$ does not depend on $h$ and on the random 
parameter $\omega$. 
%
%
%
\subsection*{The phase} 
Thanks to \cite[Equation (5.26)]{IngVog}, we may write 
\begin{equation}\label{eq:PhiIntegrale2}
\phi_{j,t, \delta} (x)
	= 
	\phi_{j,t, 0} (x) - \delta
		\int_0^t q_\omega\left(\zeta_{j,\delta}^{s,t}(x)\right) d s,
\end{equation}
where, by \cite[Equation (5.29)]{IngVog},
\begin{equation}\label{eq:PetitCoupDeStressDeLaFin}
	\mathrm{dist}_{T^*X}
	\left( \zeta_{j,\delta}^{s,t}(x),  \Phi^{-s}
		(x, d_x \phi_{j,t,0})
	\right) 
	\leq 
	\delta h^{-\beta- 2 \varepsilon}.
\end{equation}
The precise definition of $\zeta_{j,\delta}^{s,t}(x)$ is 
not important for our purposes, but we refer the reader to 
the beginning of \cite[Section 5.2.2]{IngVog} \ora{for more details.} 
Here, the covector $d_x \phi_{j,t,0}$ is the unique 
$\xi\in S_{\lambda_j}^*\wt{X}$ such that 
$\Phi^{-t}(x,\xi) \in \Lambda_{j}$ (see the discussion after (\ref{eq:Rappel}) 
for the existence and uniqueness of such a $\xi$). Furthermore, thanks to 
\cite[Lemma 5.3]{IngVog} we have, for any $\varepsilon_1>0$ small enough
\begin{equation}\label{eq:EcartGradient}
|d_x \phi_{j,t,\delta} - d_x\phi_{j,t,0}| 
	\leq C_{\varepsilon_1} \delta h^{-\beta - \varepsilon_1} 
	\leq C_{\varepsilon_1} h^{\beta - \varepsilon_1},
\end{equation}
where the second inequality comes from \eqref{eq:CondBeta}.
\\
\par
Now, thanks to \eqref{eq:ExpRate}, we have, for $j\neq j'$, 
$|d_x\phi_{j,t,0} - d_x\phi_{j',t,0}|_x \geq c h^{\beta - \varepsilon}$. 
Therefore, we deduce from  (\ref{eq:EcartGradient}) that
\begin{equation}\label{eq:EcartGradient2}
|d_x \phi_{j,t,\delta} - d_x\phi_{j',t,\delta}| \geq c' h^{\beta - \varepsilon}.
\end{equation}
To conclude this paragraph, we recall that it was shown in 
\cite[Proposition 9.2]{IngVog} that 
\begin{equation}\label{eq:QuasiCentree}
\E \left[ e^{\frac{i}{h} \phi_{j,t, \delta} (\wit{x}_{j,t})} \right] 
	= \mO(h^\Gamma).
\end{equation}
\subsection{Independence}
Recall \eqref{eq:Rappel2} and write, for each $j\in J_h$
\begin{align*}
Z_j(x) &:= s_j  \hat{b}_{j,0}(t, x) e^{\frac{i}{h}\phi_{j,t, \delta} (x)},
\end{align*}
with the convention that this quantity is equal to zero if $j\notin J_h(x)$. 

Using \eqref{eq:Rappel3}, we have for any $x\in X$
\begin{equation}\label{eq:FirstOrderL21}
\begin{aligned}
&\bigg\|g_{h,t,\omega}  - \sum_{j\in J_h} Z_j \bigg\|^2_{L^2}
= \sum_{j\in J_h} |s_j|^2  \int_X \left|b_{j,N}(t, x ; \delta,h) -
	 \hat{b}_{j,0}(t, x) \right|^2 dx  \\
	&+ \sum_{j\neq j'}  \int_X  \left(b_{j,N}(t, x ; \delta,h) - \hat{b}_{j,0}(t, x) \right) \left(b_{j',N}(t, x ; \delta,h) -  \hat{b}_{j,0}(t, x) \right)
	e^{\frac{i}{h}(\phi_{j,t, \delta} (x) - \phi_{j',t, \delta} (x))} dx.
	\end{aligned}
\end{equation}

Now, thanks to \eqref{eq:ApproxAmpl}, the integrals in the first term are 
all $\mO(h^{\Gamma})$, so that the first term is a 
$\mO(h^\gamma\|g_h\|_{L^2}^2)$. As to the second sum, an application of 
Lemma \ref{Lem:NonStat} along with (\ref{eq:EstimDerivAmpl}) and 
\eqref{eq:EcartGradient2} implies that the term $j,j'$ it is a 
$\mO(h^\infty |s_j| |s_{j'}|)$. In view of \eqref{eq:CardJh} and 
\eqref{eq:Orthogonalite}, we deduce that the second term in 
\eqref{eq:FirstOrderL21} is an $\mO(h^{\infty} \|g_h\|_{L^2}^2)$, so that
\begin{equation}\label{eq:FirstOrderL22}
\bigg\|g_{h,t,\omega}  - \sum_{j\in J_h} Z_j \bigg\|^2_{L^2}
= \mO(h^\Gamma \|g_h\|^2).
\end{equation}

%
%
%
%
Now,  from \eqref{eq:Rappel3a}, (\ref{eq:EstimDerivAmpl}), it follows that $\|Z_j\|_{L^\infty} 
\leq C \|Z_j\|_{L^2} \leq C |s_j|$, so that
\begin{equation}\label{eq:LinfiniL2}
\begin{aligned}
\sum_{j\in J_h} \|Z_j\|^2_{L^\infty} 
	&\leq C \sum_{j\in J_h} \|Z_j\|_{L^2}^2\\
	&\leq C \bigg\| \sum_{j \in J_h} Z_j \bigg\|_{L^2}^2 
		+ C \sum_{j\neq j' \in J_j} \left|\int_X Z_j(x) \overline{Z_{j'}(x)} 
			\mathrm{d}x\right|\\
			&\leq \|g_{h,t,\omega}\|^2_{L^2} + \mO(h^\Gamma \|g_h\|^2) + \mO(h^\infty \|g_h\|^2),
\end{aligned}
\end{equation}
thanks to (\ref{eq:FirstOrderL22}), and to the paragraph that precedes.
By unitarity of the propagator, we deduce that
%
%
\begin{equation}\label{eq:LinfiniL3}
\begin{aligned}
\sum_{j\in J_h} \|Z_j\|^2_{L^\infty} 
&\leq C \|g_h\|_{L^2}^2.
%
\end{aligned}
\end{equation}
Note that, here, the constant $C$ does not depend on $h$ and $\omega$. From 
now on, $x$ will be a fixed parameter, and we will write $Z_j$ for $Z_j(x)$. 
We will also write 
\begin{equation*}\label{eq:SumOfZ}
	Z:= \sum_{j\in J_h} Z_j.
\end{equation*}
\begin{lem}
The family of random variables $\left(Z_j \right)_{j\in J_h(x)}$ is independent.
\end{lem}
\begin{proof}
Let $j\neq j'\in J_h(x)$. Since $\Phi^{-t}(x, d_x\phi_{j^{(')},t,0}) 
\in \Lambda_{j^{(')}}$, we have 
\begin{equation*}
	\mathrm{dist}_{T^*X}\left(\Phi^{-t}(x, 
	d_x\phi_{j,t,0}), \Phi^{-t}(x, d_x\phi_{j',t,0}) 
	\right)> h^{\beta- \varepsilon},
\end{equation*}
and we deduce from (\ref{eq:ExpRate}) that, for any $s\in [0,t]$, 
\begin{equation*}
	\mathrm{dist}_{T^*X}\left(\Phi^{-s}(x, d_x\phi_{j,t,0}), 
	\Phi^{-s}(x, d_x\phi_{j',t,0}) \right)> c h^{\beta- \varepsilon}.
\end{equation*}
for some $c>0$. Therefore, we deduce from (\ref{eq:PetitCoupDeStressDeLaFin}) 
that, as long as $h$ is small enough, we have
\begin{equation*}
	\mathrm{dist}_{T^*X}
	\left( \zeta_{j,\delta}^{s,t}(x),  \zeta_{j',\delta}^{s,t}(x)
	\right) \geq 
	\frac{c}{2} h^{\beta- \varepsilon}.
\end{equation*}
We then deduce from (\ref{eq:PhiIntegrale2}) that $Z_j$ and $Z_{j'}$ depend 
on sets of random parameters $(\omega_i)$ that are disjoint. The result 
follows.
\end{proof}
Thanks to (\ref{eq:QuasiCentree}), we deduce that, for any $j\in J_h$, we have 
\begin{equation*}
	\E[Z_j] =\mO(h^{\Gamma}) |s_j|.
\end{equation*}
We then have 
\begin{align*} |\E[Z]| &\leq C h^\Gamma \sum_{j\in J_h} |s_j| \\
&\leq C h^\Gamma \|g_h\|_{L^2}  \times |J_h|^{1/2}\\
&\leq  C h^{\Gamma - \frac{(d-1) (\beta-\varepsilon)}{2}} \|g_h\|_{L^2}
\end{align*}
by Cauchy-Schwarz inequality and (\ref{eq:CardJh}), up to possibly shrinking 
$\varepsilon$.
\begin{corollary}\label{cor:SmallZ}
For any $\varepsilon>0$, 
\begin{equation*}
	\prob\left[ |Z|> \|g_h\|_{L^2}  \times \mO(h^{-\varepsilon}) 
	\left(1 + h^{\Gamma -  \frac{(d-1) (\beta - \varepsilon)}{2}} \right) \right] 
	= \mO(h^\infty).
\end{equation*}
\end{corollary}
\begin{proof}
Recall that, thanks to (\ref{eq:LinfiniL3}), 
$\sum_{j\in J_h} |Z_j|^2 \leq  C  \|g_h\|_{L^2}^2$, 
 with $C$ independent of $h$ and $\omega$.
For any $p\in \N$, the Marcinkiewicz-Zygmund inequality implies that 
\begin{equation*}
	\erw [|Z - \E[Z]|^p] \leq C_p \|g_h\|_{L^2}^p 
	 .
\end{equation*}
In particular, using Markov's inequality, we deduce that, with probability 
$1-\mO(h^\infty)$, we have
\begin{equation*}
	|Z - \E[Z]| 
	\leq C \|g_h\|_{L^2} h^{-\varepsilon} 
	.
\end{equation*}
The result follows.
\end{proof}
\subsection{Proof of Theorem \ref{th:Main}}
By Proposition \ref{Prop:DecompLag}, we may write
\begin{equation*}
	\psi_h = \sum_{i\in \mathcal{I}_h} g_{i,h},
\end{equation*}
with $|\mathcal{I}_h|\leq h^{(d-1)(\beta-\varepsilon-1) - \varepsilon}$, and where
\begin{equation*}
	\sum_{i\in \mathcal{I}_h} \|g_{i,h}\|_{L^2}^2 \leq C\|\psi_h\|_{L^2}^{2},
\end{equation*}
with $C$ independent of $h$. We thus have, for any $x\in X$,
\begin{align*}
\left|e^{-i \frac{t}{h} P_h^\delta} \psi_h(x)\right|^2 
	&= \left| \sum_{i\in \mathcal{I}_h}
		e^{-i \frac{t}{h} P_h^\delta} g_{i,h} (x)\right|^2\\
	&\leq |\mathcal{I}_h| \times  \sum_{i\in \mathcal{I}_h}
		\left|e^{-i \frac{t}{h} P_h^\delta} g_{i,h} (x)\right|^2
			~~\text{ by Cauchy-Schwarz}\\
	&\leq h^{(d-1)(\beta-1) - \varepsilon} \times \sum_{i\in \mathcal{I}_h}
		\left|e^{-i \frac{t}{h} P_h^\delta} g_{i,h} (x)\right|^2.
\end{align*}
Now, by Corollary \ref{cor:SmallZ}, there is an event $\Omega_h$ of 
probability $\prob(\Omega_{h})>1-\mO(h^\infty)$ such that, for all 
$\omega\in \Omega_{h}$, we have
\begin{align*}
\left|e^{-i \frac{t}{h} P_h^\delta} \psi_h(x)\right|^2 
	&\leq C h^{(d-1)(\beta -1) - \varepsilon} \times  
	\left( \sum_{i\in \mathcal{I}_h}
	\left\|g_{i,h}\right\|_{L^2}^2 \right) \times 
		 \left( 1 + h^{2\Gamma -  (d-1) (\beta - \varepsilon)} \right) \\
	&\leq  C \|\psi_h\|_{L^2}^2 \left( h^{(d-1)(\beta - \varepsilon-1) 
		- \varepsilon} +  h^{2\Gamma -(d-1)- C \varepsilon}\right)\\
	&\leq C \|\psi_h\|_{L^2}^2   h^{2\Gamma' -(d-1)- C \varepsilon} ,
\end{align*}
with $\Gamma'$ as in (\ref{eq:DefGammaPrime}).
\\
\par
Consider a collection of \emph{mesh points} $(x_n)_{n=1..., N_h}\subset X$ 
with $N_h=N_h(M)$ bounded polynomially in $h^{-1}$, such that any point $x\in X$ 
is at a distance at most $h^{M}$, $M>0$, from some mesh point. By a union 
bound it follows that with probability $1-\mO(h^\infty)$, we have that 
for all $n=1,\dots,N_h$ 
\begin{align*}
\left|e^{-i \frac{t}{h} P_h^\delta} \psi_h(x_n)\right|
	&\leq C h^{\Gamma' -\frac{(d-1)}{2}- C \varepsilon} \|\psi_h\|_{L^2}.
\end{align*}
\par
Since $\psi_h$ is an eigenfunction of $-h^2\Delta_g$, we know 
by elliptic regularity that $\psi_h$ is smooth. By \eqref{eq:SchTyp5}, 
it follows that $e^{-i \frac{t}{h} P_h^\delta} \psi_h \in H^s_h(X)$, 
for any $s\geq 0$, and 
\begin{equation*}
\|e^{-i \frac{t}{h} P_h^\delta} \psi_h \|_{H^s_h(X)}
	\leq \mO(1)\|\psi_h \|_{H^{s}_h(X)}.
\end{equation*}
For any $s\geq0$, standard Sobolev estimates show that for any 
$u\in H^s_h(X)$, $s>d/2$, 
\begin{equation*}
	\|u\|_{L^\infty(X)} 
	\leq \mO_s(1)h^{-d/2}\|u\|_{H^s_h(X)}.
\end{equation*}
Hence, 
\begin{equation*}
\begin{split}
	\|\nabla e^{-i \frac{t}{h} P_h^\delta} \psi_h \|_{L^\infty(X)} 
	&\leq \mO_s(1)h^{-d/2}\|\nabla e^{-i \frac{t}{h} P_h^\delta} \psi_h \|_{H^s_h(X)}\\
	&\leq \mO_s(1)h^{-d/2-1}\|e^{-i \frac{t}{h} P_h^\delta} \psi_h \|_{H^{s+1}_h(X)}\\
	&\leq \mO_s(1)h^{-d/2-1}\| \psi_h \|_{H^{s+1}_h(X)}\\
	&\leq \mO_s(1)h^{-d/2-1}\| \psi_h \|_{L^2(X)},
\end{split}
\end{equation*}
where in the last line we used again elliptic regularity. 
\par
Taking $M>0$ sufficiently large (i.e. a sufficiently fine mesh), 
we deduce that, with probability $1-\mO(h^\infty)$,
\begin{equation*}
	\left\|e^{-i \frac{t}{h} P_h^\delta} \psi_h \right\|_{L^\infty} 
	\leq C h^{\Gamma' -\frac{(d-1)}{2}- C \varepsilon}  \|\psi_h\|_{L^2},
\end{equation*}
and the result follows by taking $\varepsilon$ smaller, concluding 
the proof of Theorem \ref{th:Main}.
\appendix  
\section{Review of semiclassical analysis}\label{sec:Appendix}
In this section we briefly recall basic notions and results from semiclassical 
analysis. We refer the reader to the books \cite{DiSj,Zw12} for a detailed 
introduction to the subject. 

\subsection{Non-stationary phase}
We use several times the following lemma, whose proof is analogous to that 
of \cite[Lemma 3.14]{Zw12}.
\begin{lem}\label{Lem:NonStat}
Let $h\in]0,1]$, let $X$ be a Riemannian manifold, and let $a(\cdot;h) \in C_c^\infty(X)$, 
$\varphi(\cdot;h) \in C^\infty(X)$. 
Suppose that, $a(\cdot; h)$, $\varphi( \cdot ; h)$ are bounded independently 
of $h$, along with all their derivatives, and that the support of $a(\cdot ; h)$ 
is contained in a compact set $K$ independent of $h$. Suppose that there exist $c>0$ and 
$\varepsilon\in ]0,1/2[$ such that, for all $x\in K$ and all $h\in ]0,1]$, 
we have $|d_x \varphi(x; h)| \geq c h^{1/2 - \varepsilon}$. Then, 
\begin{equation*}
	\int_X a(x;h) e^{\frac{i}{h} \varphi(x;h)} dx = \mO(h^\infty).
\end{equation*}
More generally,  if $a$ and $\varphi$ depend on an extra parameter 
 $\lambda>0$, with $|d_x \varphi(x; h)| \geq c\lambda h^{1/2 - \varepsilon}$ 
 but all derivatives of $a$ and $\varphi$ bounded independently of $h$ and 
 $\lambda$, then 
 \begin{equation*}
 	\int_X a(x;h,\lambda) e^{\frac{i}{h} \varphi(x;h,\lambda)} dx 
 		= \mO\!\left(\left(\frac{h^{\varepsilon}}{ \lambda}\right)^\infty\right).
 \end{equation*}
\end{lem}
\subsection{Pseudodifferential calculus} \label{sec:Pseudo}
Let $X$ be a smooth $d$-dimensional Riemannian manifold without boundary. 
Let $T^*X$ denote the cotangent bundle and let $\overline{T}^*X$ denote 
the fiber-radial compactification of $T^*X$, see \cite[Section 2]{Va13} and 
the book \cite[Section E.1.3]{DyZw19}.  
\\
\\
\textbf{Symbol Classes.} For $k\in\R$ and $\eta\in [0,1/2[$ we say that a smooth function 
$a(x,\xi;h)\in C^{\infty}(T^*X\times (0,1])$ lies in the \emph{symbol 
class} $S^k_\eta(X)$ if and only if for all compact sets $K\subset X$ and 
all multiindices $\alpha,\beta\in \N^d$ there exists a constant 
$C_{\alpha,\beta,K}>0$ such that 
\begin{equation}\label{app:SC1}
	\sup_{x\in K}| \partial^\alpha_x \partial_\xi^\beta 
	a(x,\xi;h)| \leq C_{\alpha,\beta,K} h^{-\eta(|\alpha|+|\beta|)}
	\langle\xi\rangle^{k-|\beta|}.
\end{equation}
Here $\langle\xi\rangle:= (1+|\xi|_g^2)^{1/2}$ denotes the ``Japanese 
bracket''. Note that this symbol class is independent of the choice 
of local coordinates. We define a residual class of symbols by 
$S^{-\infty}_\eta(X) := \bigcap_{k\in \R} S^k_\eta(X)$ and
$h^\infty S^{-\infty}_\eta(X)$ by 
\begin{equation*}
	a \in h^\infty S^{-\infty}_\eta(X) \quad \Longleftrightarrow 
	\quad 
	\forall \alpha,\beta \in \N^d , ~K\Subset X, ~N>0:~~
	\partial^\alpha_x \partial_\xi^\beta
	a(x,\xi;h) = \mO_{\alpha,\beta,K,N}(h^N \langle \xi\rangle^{-N}),
\end{equation*}
uniformly when $x$ varies in $K$. 
\\
\par
We will also consider the symbol classes $S^{\comp}_\eta(T^*X)$ consisting 
of compactly supported functions $a(x,\xi;h)\in C^{\infty}(T^*X\times (0,1])$ 
satisfying \eqref{app:SC1}, with a support bounded independently of $h$. 
Hence, $S^{\comp}_\eta(T^*X)\subset S^k_\eta(X)$ for all $k\in\R$. 
When $\eta=0$, then we write $S^{\comp}(T^*X) = S^{\comp}_0(T^*X)$ and 
$S^k(X)=S^k_0(X)$ for simplicity. If $U$ is an open set of $T^*X$, we will 
sometimes write $S^{\comp}_\eta(U)$ for the set of functions $a$ in 
$S^{\comp}_\eta(T^*X)$ such that for any $h\in ]0,1]$, $a$ has its 
support in $U$. 
\\
\\
\textbf{Pseudodifferential operators and quantization.} 
A linear continuous map $R=R_h: \mathcal{E}'(X) \to C^{\infty}(X)$ 
is called \emph{negligible}, or in the \emph{residual class} $h^\infty \Psi^{-\infty}$, 
if its distribution kernel $K_R$ is smooth 
and each of its $C^\infty(X\times X)$ seminorms is $\mO(h^\infty)$, i.e. it satisfies 
\begin{equation*}
	\partial_x^\alpha \partial^\beta_y K_R(x,y)= \mO(h^\infty), 
\end{equation*}
for all $\alpha,\beta\in \N^d$, when expressed in local coordinates.
\\
\par
A linear continuous map $P_h: C_c^{\infty}(X)\to \mathcal{D}'(X)$ is called a 
\emph{semiclassical pseudo-differential operator} belonging to the space 
$\Psi_{h,\eta}^{m}$ if and only if the following two conditions hold: 
\begin{enumerate}
	\item $\phi P_h \psi$ is negligible for all 
	$\phi,\psi \in C^\infty_c(X)$ with $\supp \phi \cap \supp \psi = \emptyset$;
	\item for every cut-off chart $(\kappa,\chi)$ there exists a symbol 
	$p_{\kappa}\in S^m_\eta(T^*\R^d)$ such that
	\begin{equation}\label{app:PDO1} 
		\chi P_h\chi 
		= \chi\kappa^*  \Op_h^\mathrm{w}(p_{\kappa})(\kappa^{-1})^*\chi.
	\end{equation}
\end{enumerate}
Here, $\kappa: X\ni U \to V\subset \R^d$ is a diffeomorphism between open sets 
and $\chi\in C_c^\infty(U)$. We refer to the pair $(\kappa,\chi)$ as a 
\emph{cut-off chart}. In \eqref{app:PDO1} we use the semiclassical Weyl 
quantization of the symbols $p_\kappa$ defined by
\begin{equation}\label{app:PDO2} 
	\Op_h^\mathrm{w}(p_\kappa) u(x) = \frac{1}{(2h\pi)^d} \iint_{\R^{2d}} 
	\e^{\frac{i}{h}(x-y)\cdot \xi} 
	p_\kappa\!\left(\frac{x+y}{2},\xi;h\right)u(y) dy d\xi, 
	\quad u\in C_c^\infty(V),
\end{equation}
seen as an oscillatory integral.
\\
\par
We have the surjective semiclassical principal symbol map 
\begin{equation}\label{app:PDO3} 
	\sigma: \Psi_{h,\eta}^{m}(X) \rightarrow S^m_\eta(X)/h^{1-2\eta}S^{m-1}_\eta(X)
\end{equation}
whose kernel is given by $h^{1-2\eta}\Psi_{h,\eta}^{m-1}$ and its right inverse 
is given by a non-canonical quantization map 
\begin{equation}\label{app:PDO4} 
	\Op_h: S^m_\eta(X)\rightarrow \Psi_{h,\eta}^{m}(X) .
\end{equation} 
Such a (non-intrinsic) quantization map can be defined for example as follows: let 
$(\phi_j,\chi_j)_j$ be a countable family of cut-off charts whose 
domains $U_j$ cover $X$ and such that $\sum_j \chi_j =1$ on $X$. 
Let $\chi_j'\in C_c^\infty(U_j;[0,1])$ be equal to $1$ near 
$\supp \chi_j$. Then, given $a\in S^m_\eta(X)$ we define 
\begin{equation}\label{app:PDO4.1} 
	\Op_h(a):= \sum_j \chi_j'\phi_j^* 
		\Op_h^\mathrm{w}((\widetilde{\phi}_j^{-1})^*(\chi_ja))
		(\phi_j^{-1})^* \chi_j',
\end{equation} 
where $\widetilde{\phi}_j:T^*U_j\to T^*V_j$ is a lift of the 
chart $\phi_j$ defined by $\widetilde{\phi}_j(x,\xi)= \ora{(}(x,(d\phi(x))^{-1}\ora{)^T}\xi)$.   
Notice that $\Op_h(a)$ is in particular properly supported. 
Furthermore, we have that 
\begin{equation}\label{app:PDO5} 
	\Psi_{h,\eta}^{m}(X) = \Op_h(S^m_\eta(X)) +h^\infty\Psi^{-\infty}.
\end{equation} 
We see that any operator in $\Psi_{h,\eta}^{m}(X)$ can represented 
by a properly supported operator in $\Psi_{h,\eta}^{m}(X)$ up to a 
negligible term in $h^\infty\Psi^{-\infty}$.  
\\
\par
Compositions of pseudo-differential operators are pseudo-differential operators 
and we have for properly supported $A\in \Psi_\eta^m(X)$, 
$B\in \Psi_\eta^{m'}(X)$ that $A\circ B \in \Psi_\eta^{m+m'}(X)$ and 
\begin{equation}\label{app:PDO5.1} 
\begin{split}
	\sigma_h(A\circ B) = \sigma_h (A) \sigma_h(B) ,\\
	\sigma_h([A,B]) =-ih\{\sigma_h (A), \sigma_h(B)\},\\
\end{split}
\end{equation}
where $\{\cdot, \cdot\}$ denotes the Poisson bracket with respect to the natural 
symplectic structure on $T^*X$. 
\\
\\
\textbf{Wavefront set.} For $a\in S^{k}_\eta(T^*X)$ we define its essential support 
$\esupp a\subset \overline{T}^*X$ as follows: a point $\rho \in \overline{T}^*X$ is 
not contained in $\esupp a$ if there exists a neighborhood $U$ of $\rho$ in $ \overline{T}^*X$ 
such that for all $\alpha,\beta \in \N^d$, $N>0$ there exists a constant $C_{\alpha,\beta,N}>0$ 
such that 
\begin{equation*}
	| \partial^\alpha_x \partial_\xi^\beta 
	a(x,\xi;h)| \leq C_{\alpha,\eta,N} h^{N}\langle\xi\rangle^{-N}, 
	\quad (x,\xi) \in U\cap T^*X. 
\end{equation*}
For $P_h\in \Psi_{h,\eta}^{m}$ we define the 
\emph{wavefront set} $\WF_h(P_h) \subset \overline{T}^*X$ as follows: 
a point $(x,\xi)\in \overline{T}^*X$ does not lie in $\WF_h(P_h)$ if and only 
if for each cut-off chart $(\phi,\chi)$ such that $x$ is contained in the 
domain of $\phi$, we have that $(\phi(x),\ora{(}(d\phi(x))\ora{^{-1})^{T}}\xi)\notin \esupp p_\kappa$, 
with $p_\kappa$ defined as in \eqref{app:PDO1}. 
\\
\par
If $A\in \Psi_{h,\eta}^{m}$, $B_h \in \Psi_{h,\eta'}^{m'}$ are properly 
supported, then 
\begin{equation}\label{eq:CompoWF}
\WF_h (A B) \subset \WF_h (A) \cap \WF_h (B).
\end{equation}
In particular, if $b \equiv 1$ on $\WF_h(A)$, we have
\begin{equation}\label{eq:MicrolocallyId}
A = \Op_h(b) A + h^\infty \Psi^{-\infty}.
\end{equation}
\par
Compactly supported pseudo-differential operators in $\Psi^m_\eta(X)$ with compact wavefront sets 
in $T^*X$ are called \emph{compactly microlocalized} and we will denote this 
class by $\Psi^\comp_\eta(X)$. Notice that $\Psi^\comp_\eta(X) \subset \Psi^m_\eta(X)$ 
for all $m\in \R$. Furthermore, 
\begin{equation}\label{app:PDO6} 
	\Psi_{h,\eta}^{\comp}(X) = \Op_h(S^\comp_\eta(X)) +h^\infty\Psi^{-\infty}.
\end{equation} 
When we make use of this identity for an $A\in\Psi_{h,\eta}^{\comp}(X)$ 
then we call an $a\in S^\comp_\eta(X)$, such that $A = 
\Op_h(a) +h^\infty\Psi^{-\infty}$, a \emph{full symbol} of $A$ (though this 
notion is not intrinsic and depends on a choice of quantization). 
For $A=\Op_h(a)\in \Psi^{\comp}_h(X)$, we have
\begin{equation}\label{app:PDO6.1} 
	\WF_h(A) = \esupp a,
\end{equation}
noting that this does not depend on the choice of the
quantization. When $K$ is a compact subset of $T^*X$ and $\WF_h(A)\subset K$, 
we will sometimes say that $A$ is \emph{microsupported} inside $K$.
\\
\par 
The principal symbol map
\begin{equation*}
	\sigma_h: \Psi_{h,\eta}^{\comp}(X) \rightarrow S^{\comp}_\eta (X)
	/(h^{1-2\eta}S^{\comp}_\eta (X) + h^\infty S^{-\infty}(X)) 
\end{equation*}
is surjective with kernel given by $h^{1-2\eta}\Psi_{h,\eta}^{\comp}(X)$. 
Thanks to (\ref{eq:CompoWF}), the composition of two compactly 
microlocalized operators is still compactly microlocalized, i.e. 
\begin{equation*}
	A,B\in\Psi_\eta^\comp(X) \quad \Longrightarrow \quad 
	A\circ B \in \Psi_\eta^\comp(X),
\end{equation*}
and \eqref{app:PDO5.1} is still valid. Moreover, the standard composition 
formulas of symbols in $\R^{2d}$, see for instance \cite[Proposition 7.6]{DiSj} 
show also that when $A\in \Psi_0^k(X)$ and $B\in \Psi_\eta^\comp (X)$ then 
$A\circ B \in \Psi_\eta^\comp (X)$ and 
\begin{equation}\label{app:PDO7} 
	\sigma_h(AB) \equiv \sigma_h(A)\sigma_h(B) \quad \mathrm{mod} 
	~~h^{1-\eta}S^{\comp}_\eta(X) + h^{-\infty}S^\infty(X).
\end{equation}
\subsection{Microlocalization}\label{sec:Microloc}
Given the notion of semiclassical wavefront set, it is natural 
to consider operators and their properties \emph{microlocally}. 
Let $X_1,X_2$ be two smooth $d$-dimensional manifolds and 
let $h\in]0,h_0]$. We say that a family of distributions 
$u_h\in \mathcal{D}'(X_1)$ is $h$-tempered if for each $\chi\in C^\infty_c(X_1)$ 
there exist constants $N\geq 0$ and $C>0$ such that 
\begin{equation*}
	\| \chi u_h \|_{H^{-N}_h(X_1)} \leq C h^{-N},
\end{equation*}
where $\| \cdot\|_{H^{s}_h(X_1)}$, $s\in\R$, denotes the 
semiclassical Sobolev norm of order $s$, see e.g. 
\cite[Defintion E.19]{DyZw19}. 
We say that a family of operators $T_h:C_c^\infty(X_2)\to 
\mathcal{D}'(X_1)$ is $h$-tempered if the family of associated 
Schwartz kernels $K_{T_h}\in \mathcal{D}'(X_1\times X_2)$ is 
$h$-tempered. 
\par
For open sets $V\Subset T^*X_1$ and 
$U\Subset T^*X_2$, the operators \emph{defined microlocally} 
near $V\times U$ are given by the equivalence classes of tempered operators 
defined by the following relation: 
$T\sim T'$ if and only if there exists open sets 
$\widetilde{V}\Subset T^*X_1$, $\widetilde{U}\Subset T^*X_2$ with $U\Subset\widetilde{U}$, 
$V\Subset\widetilde{V}$ such that, for any $A\in\Psi_h^\comp(X_1)$, $B\in \Psi_h^\comp(X_2)$ with
$\WF_h(A)\subset\widetilde{V}$, 
$\WF_h(B)\subset\widetilde{U}$, we have
\begin{equation}\label{app:Mloc2} 
	A(T-T')B=\mO(h^\infty)_{\Psi^{-\infty}}.
\end{equation}
\par
For two such operators we say that $T=T'$ \emph{microlocally} near $V\times U$. 
Similarly, we say that $S=T^{-1}$ microlocally near $V\times V$ (and thus, that $S$ is \emph{microlocally invertible} near $V\times V$), if 
$ST=Id$ microlocally near $U\times U$, and $TS=Id$ microlocally near $V\times V$.

\subsection{Quantizing symplectic maps}\label{sec:FIO}
Let $Y, Y'$ be either compact Riemannian manifolds or $\R^d$. 
Let $U\Subset T^*Y ,U'\Subset T^*Y'$ be bounded open sets, and let 
$\kappa : U\to U'$ be a symplectomorphism. We say that an ($h$-dependent)  
operator $T: L^2(Y)\to L^2(Y')$ quantizes $\kappa$ microlocally 
in $U\times \kappa(U)$ if, for any $a\in C_c^\infty(U)$, we have
 \begin{equation}
 T \Op_h(a) = \Op_h(a') T
 ~~\text{microlocally near } U\times\kappa(U), 
 \quad 
 a'	:= a \circ \kappa^{-1} + \mO_{S^0}(h).
 \end{equation}
\\
\textbf{Quantizing coordinate changes.}
As a first example, consider a point $x_0\in X$, and a coordinate 
chart $\varphi : V_{x_0} \to U_{x_0}\subset \R^d$ 
defined in a neighborhood $V_{x_0}$ of $x_0$.  Then $\varphi$ can be 
uniquely lifted to a symplectic map $\kappa_\varphi : T^* V_{x_0} 
\to T^*U_{x_0}\subset\R^{2d}$, which is bijective on its image.
Let $V'_{x_0}\Subset V_{x_0}$ be an open 
set containing $x_0$, and let $\chi\in C_c^\infty(T^*V_{x_0})$ 
with $\chi\equiv 1$ on $V'_{x_0}$. Then the map 
$T_\varphi: L^2(X) \to L^2(\R^d)$ given by 
\begin{equation}\label{eq:CoordChange}
	T_\varphi f=  (\kappa_\varphi^{-1})^*(\chi f)
\end{equation}
is a quantization of $\kappa$ in $V'_{x_0}\times \kappa (V'_{x_0})$. Note also that 
$T_\varphi: C^\infty(X) \to C^\infty_c(\R^d)$ and that $T_\varphi=\mO_s(1): 
H_h^s(X)\to H^s(\R^d)$, $s\in\R$. Similarly, we can 
consider $S_\varphi:L^2(\R^d)\to L^2(X)$ defined by 
$S_\varphi f = \chi (\kappa_\varphi)^*f$. One can easily 
check that this map has the similar properties. 
We refer to \cite{Zw12} for more details.  
\\
\\
\textbf{Quantizing symplectomorphisms locally.} 
Following \cite[Section 4.1]{NZ}, we consider the case $Y=Y'=\R^d$, with 
symplectomorphisms defined from graphs.  Namely, we suppose that $\kappa$ 
is a symplectomorphism defined in a neighborhood of the origin, with 
$\kappa(0,0)=(0,0)$, and such that 
\begin{equation}\label{eq:SymplectoGraph}
T^*\mathbb{R}^d\times T^*\mathbb{R}^d \ni (x^1,\xi^1; x^0,\xi^0)
	\mapsto (x^1,\xi^0)\in \mathbb{R}^d\times \mathbb{R}^d,~~ (x^1,\xi^1)
	= \kappa(x^0,\xi^0),
\end{equation}
is a diffeomorphism near the origin. Note that this is equivalent to asking that
\begin{equation}\label{eq:Block}
\text{the } n\times n \text{ block } (\partial x^1/\partial x^0) 
\text{ in the tangent map } d\kappa(0,0) \text{ is invertible}.
\end{equation}
It then follows that there exists a unique function 
${\psi}\in C^\infty(\mathbb{R}^d\times \mathbb{R}^d)$ such that for 
$(x^1,\xi^0)$ near $(0,0)$,
\begin{equation*}
\kappa(\partial_\xi\psi(x^1,\xi^0),\xi^0)
	=(x^1,\partial_x\psi(x^1,\xi^0)),
	~~ \det \partial^2_{x,\xi}\psi\neq 0 \text{ and } \psi(0,0)=0.
\end{equation*}
The function $\psi$ is said to \emph{generate} the transformation $\kappa$ near $(0,0)$.
\par
Consider a function $\alpha\in S^{\comp}(\mathbb{R}^{2d})$.
Then the operator $T : L^2(\R^d) \to L^2(\R^d)$ given by
\begin{equation}\label{eq:LocalFIO} 
T u(x^1):= \frac{1}{(2\pi h)^d}\iint_{\mathbb{R}^{2n}} e^{i(\psi(x^1,\xi^0)
	-\langle x^0,\xi^0 \rangle/h} \alpha (x^1,\xi^0;h) 
	u(x^0) dx d x^0 \ora{d}\xi^0,
\end{equation}
is a quantization of $\kappa$, and is microlocally invertible in a 
neighborhood of $(0,0)\times (0,0)$, provided $\alpha(x^1,\xi^0;h) = 
|\det \psi_{x,\xi}''(x^1, \xi^0)|^{1/2} + \mO(h)$ close to $(0,0)$. 
Furthermore, any quantization of $\kappa$ that is microlocally invertible 
in a neighborhood of $(0,0)\times (0,0)$ is of the form (\ref{eq:LocalFIO}), 
microlocally in a neighborhood of $(0,0)\times (0,0)$. More precisely, 
there exists an operator $S:L^2(\R^d) \to L^2(\R^d)$ of the form 
\eqref{eq:LocalFIO} quantizing $\kappa^{-1}$, such that 
\begin{equation}\label{eq:LocalFIO2} 
	ST = I, \quad TS = I, \quad \text{microlocally near } 
	(0,0)\times (0,0). 
\end{equation}

An operator of the form (\ref{eq:LocalFIO}) is called an $h$-Fourier Integral 
Operator associated to $\kappa$. Moreover, an operator $T$ of the 
form \eqref{eq:LocalFIO} satisfies 
\begin{equation*}
	T=\mO_{s}(1):~H_h^s(\R^d)\to H_h^s(\R^d), \quad s\in\R,
\end{equation*}
and for any $b\in S^\comp (\R^d)$ we have that 
\begin{equation}\label{eq:LocalFIO3} 
	S\Op_h^\mathrm{w}(b)T = \Op_h^\mathrm{w}(c) + \mO(h^\infty)_{\Psi^{-\infty}}
\end{equation}
for some $c\in S^\comp (\R^d)$, with $\WF_h(S\Op_h^\mathrm{w}(b)T ) \subset 
\kappa^{-1}(\WF_h(\Op_h^\mathrm{w}(b)))$. The principal symbol of $c$ satisfies 
$c_0 = \kappa^* b_0$ near $(0,0)$.  
\\
\\
Let $X$ be a compact smooth Riemannian manifold, let $\rho_0\in T^*X$, 
let $U\subset T^*X$ be an open neighborhood of $\rho_0$, and let 
$U'\subset T^*\R^d$ be an open neighborhood of $(0,0)\in T^*\R^d$. 
Suppose that $\kappa:U\to U'$ is a symplectomorphism such that 
$\kappa(\rho_0)=(0,0)$. Upon potentially shrinking $U$ we may suppose 
that it is in the domain of some local coordinate chart 
$\varphi:U\to V\subset T^*\R^d$. Then $\widetilde{\kappa} = \kappa\circ 
\varphi^{-1}$ is a symplectomorphism $V\to U'$. Let $\widetilde{T}$ 
be $h$-Fourier integral operator of the form \eqref{eq:LocalFIO} quantizing 
$\widetilde{\kappa}$, let $\widetilde{S}$ be as in \eqref{eq:LocalFIO2}, 
and let $T_\varphi$ be as in \eqref{eq:CoordChange}, Then 
\begin{equation}
	T := \widetilde{T}\circ T_\varphi: ~L^2(X)\to L^2(\R^d)bl
\end{equation}
quantizes $\kappa$ microlocally near $(0,0)\times \rho_0$. Furthermore, 
$S:=S_\varphi \circ \widetilde{S}$ is a microlocal inverse of $T$ 
near $(0,0)\times (0,0)$. In fact, $T$ is an $h$-Fourier integral 
operator associated with $\kappa$, and $S$ is an 
$h$-Fourier integral operator associated with $\kappa^{-1}$. 
Moreover, $T=\mO_{s}(1):H_h^s(X)\to H_h^s(\R^d)$ and 
$S=\mO_{s}(1):H_h^s(\R^d)\to H_h^s(X)$, and for any 
$b\in S^\comp (X)$ we have that 
\begin{equation}\label{eq:LocalFIO3b} 
	S\Op_h(b)T = \Op^{\mathrm{w}}_h(c) + \mO(h^\infty)_{\Psi^{-\infty}}
\end{equation}
for some $c\in S^\comp (X)$, with $\WF_h(S\Op_h^\mathrm{w}(b)T ) \subset 
\kappa^{-1}(\WF_h(\Op_h^\mathrm{w}(b)))$. The principal symbol of $c$ satisfies 
$c_0 = \kappa^* b_0$ near $(0,0)$. 
\\
\\
\textbf{Action of Fourier integral operators on Lagrangian states.} 
Now, let us state a lemma which was proven in \cite[Lemma 4.1]{NZ}, and which 
describes the effect of a Fourier integral operator of the form (\ref{eq:LocalFIO}) 
on a Lagrangian distribution which projects on the base manifold without caustics.
\begin{lem}\label{lem:FIOonLag}
Consider a Lagrangian $\Lambda_0=\{(x_0,\phi_0'(x_0)); x\in \Omega_0\},
\phi_0\in C_b^\infty(\Omega_0)$, contained in a small neighborhood of the origin
$V\subset T^*\mathbb{R}^d$, a symplectomorphism $\kappa$ and a function $\psi$ 
such that $\kappa$ is generated by $\psi$ near $V$. We assume that
\begin{equation*}
\kappa(\Lambda_0)=\Lambda_1 = \{(x,\phi_1'(x)); x\in
\Omega_1\},~~\phi_1\in C_b^\infty(\Omega_1).
\end{equation*}
Then, for any symbol $a\in S^{\comp}(\Omega_0)$,
the application of a Fourier integral operator $T$ of the form (\ref{eq:LocalFIO}) 
to the Lagrangian state
\begin{equation*}
a(x) e^{i\phi_0(x)/h}
\end{equation*}
associated with $\Lambda_0$ can be expanded, for any $L > 0$, into
\begin{equation}\label{eq:FIOOnLagrangian}
T (a e^{i\phi_0/h})(x) = e^{i\phi_1(x)/h} \Big{(} \sum_{j=0}^{L-1} b_j(x)
h^j+ h^L r_L(x,h) \Big{)},
\end{equation}
where $b_j\in S^{comp}$, and for any $\ell\in \mathbb{N}$, we have
\begin{equation*}
\begin{aligned}
\|b_j\|_{C^\ell(\Omega_1)}&\leq C_{\ell,j}
\|a\|_{C^{\ell+2j}(\Omega_0)},~~~~0\leq j\leq L-1,\\
\|r_L(\cdot,h)\|_{C^\ell(\Omega_1)}&\leq C_{\ell,L}
\|a\|_{C^{\ell+2L+d}(\Omega_0)}.
\end{aligned}
\end{equation*}
The constants $C_{\ell,j}$ depend only on $\kappa$, $\alpha$ and
$\sup_{\Omega_0} |\partial^\beta \phi_0|$ for $0<|\beta|\leq 2\ell +j$.
Furthermore, if we write $g :\Omega_1\ni x \mapsto g(x):= \pi\circ \kappa^{-1} 
(x,\phi_1'(x))\in \Omega_0$, the principal symbol $b_0$ satisfies
\begin{equation*}
b_0(x^1)= e^{i\theta/h}\frac{\alpha_0(x^1,\xi^0)}{|\det \psi_{x,\xi}(x^1,\xi^0)|^{1/2}}
	|\det dg(x^1)|^{1/2} a\circ g(x^1),
\end{equation*}
where $\xi_0= \phi_0'\circ g(x^1)$ and where $\theta\in \mathbb{R}$.
\end{lem}
\subsection{An exotic version of Egorov's Theorem}\label{sec:egorov}
To finish, we present a version of Egorov's theorem for symbols in an 
exotic symbol class. This theorem, which relates 
the evolution of quantum and classical observables, is usually stated when 
the Hamiltonian generating the classical evolution is independent of $h$, 
see for instance \cite[Theorem 11.1]{Zw12}. 
In our case the Hamiltonian \eqref{eq:Hamiltonian} depends on $h$, and so 
we need this more general version. 
\\
\par
Let $X$ be a compact smooth Riemannian manifold, and recall that the 
operator $P_h^\delta$ given in (\ref{eq:SchroedingerOp}) is 
of the form
%
%
\begin{equation*}
	P_h^\delta := -\frac{h^2}{2} \Delta + \delta Q, \quad 0\leq \delta \ll 1,
\end{equation*}
with $Q\in \Psi^{-\infty}_\beta(X)$ self-adjoint. The operator $P_h^\delta$ 
does thus have full symbol $p(x,\xi;\delta)= \frac{1}{2}|\xi|_g^2 + \delta q$, 
$q\in S_\beta^{-\infty}(T^*X)$. 

Recall from \eqref{eq:CondBeta} that we assume that there exists 
$0<\varepsilon_0< \frac{1}{4}$ and $h_0>0$ such that for all $h\leq h_0$, we have
\begin{equation}\label{eq:DeltCondition}
\delta h^{-2\beta - \varepsilon_0} \leq 1.
\end{equation}
By the Kato-Rellich theorem we know that the operator $P_h^\delta$ is selfadjoint 
with domain $H^2_h(X)$. By Stone's Theorem (see for instance \cite[Theorem C.13]{Zw12}),
we then know that it induces a strongly continuous unitary group
\begin{equation}\label{eq:SchTyp3}
	U_\delta(t)= \e^{-i\frac{t}{h}P_h^\delta} =\mO(1): ~L^2(X) \to L^2(X), \quad 
	t\in \R. 
\end{equation}
Since the Laplacian $-\Delta_g$ is a positive elliptic second order 
differential operator on $X$,  we can equip the Sobolev 
space $H^k_h(X)$ with the norm $\|(1-h^2\Delta_g)^{k/2} f\|_{H^0}$. 
Since $P_h^\delta$ and $U_\delta(t)$ commute, we see that 
\begin{equation}\label{eq:SchTyp4}
\begin{split}
	\| (1-h^2\Delta_g)U_\delta(t)f \|_{H_h^{k}} 
	&\leq \|U_\delta(t)(1+P_h^\delta)f \|_{H_h^{k}} + \|\delta Q f\|_{H_h^{k}} 
	\\
	&\leq \|U_\delta(t)(1+P_h^\delta)f \|_{H_h^{k}} 
		+ \mO_k(\delta)\|f\|_{H_h^{0}} 
	\\
	&\leq \|U_\delta(t)(1-h^2\Delta_g)f \|_{H_h^{k}} 
	+\|U_\delta(t)\delta Q f \|_{H_h^{k}} 
		+ \mO_k(\delta)\|f\|_{H_h^{0}},
\end{split}
\end{equation}
where in the last line we used that since $Q$ is compactly microlocalized, 
it follows that $Q=\mO_N(1):H^{-N}_h\to H^{N}_h$. When $k=0$, then 
\eqref{eq:SchTyp3} and \eqref{eq:SchTyp4} yield that 
$U_\delta(t)=\mO(1): ~H^2_h(X) \to H^2_h(X)$. Iterating this argument, 
we obtain that $U_\delta(t)=\mO_n(1): ~H^{2n}_h(X) \to H^{2n}_h(X)$, and we 
deduce by duality and interpolation that 
\begin{equation}\label{eq:SchTyp5}
	U_\delta(t)=\mO_s(1): ~H^{s}_h(X) \to H^{s}_h(X), \quad s\in \R. 
\end{equation}
\par

Recall from \eqref{eq:HamiltonianVF} and \eqref{eq:HamiltonianFlow} that $H_p$ denotes the Hamilton vector flow induced 
by $p$, and that  $\Phi_\delta^t$ denotes the associated Hamilton flow.
Let $K\subset T^*X$ be 
a compact set and let $T>0$, then there exists a compact 
set $K_T\subset T^*X$, independent of $h$ and $\delta$, such that $\Phi^t_\delta (K)\subset K_T$ for all 
$t\in [-T,T]$. For each $k$, fix a norm 
$\|\cdot \|_{C^k(U;K_T)}$, by covering $U$ and $K_T$ in finitely many local 
coordinate charts, for the space $C^k(U;K_T)$ of $k$ times continuously 
differentiable functions on $U$ with values in $K_T$. Then, by mimicking 
the proof of \cite[Lemma 4.1]{IngVog} for finite times, we find that for 
each $k\in\N$ 
\begin{equation}\label{eq:SchTyp6}
	\|\Phi^t_\delta  \|_{C^k(U;K_T)} \leq \mO_{k,T}(1)(1+\delta h^{-\beta(k+1)})
\end{equation}
uniformly in $t\in [-T,T]$.
%
%

\begin{prop}\label{app:prop.Egorov}
	Let $\beta \in[0,1/2[$. Then for each $t\in\R$ and each $A\in \Psi_\beta^\comp(X)$ 
	there exists a $A^t_\delta \in \Psi_\beta^\comp(X)$, such that 
	\begin{equation}\label{app:eg0}
		U_\delta(-t)AU_\delta(t) = A^t_\delta + h^\infty \Psi^{-\infty}.
	\end{equation}
	Moreover, $\WF_h(A^t_\delta)\subset \Phi^t_\delta(\WF_h(A))$ and 
	$\sigma(A^t_\delta) = \sigma(A)\circ \Phi^t_\delta + \mO(h^{1-2\beta})
	\in S^\comp_\beta(X)$. 
\end{prop}
\begin{proof}
We will mimic the proof of \cite[Theorem 11.1]{Zw12}. 
\\
\\
1. 
Let $a\in S_\beta^\comp$ be the full symbol of $A$ and define 
\begin{equation}\label{app:eg1}
	a^t:= a \circ \Phi^t_\delta.
\end{equation}
Let $t\in [-T,T]$, and notice that $a^t$ is compactly supported in 
$\Phi^{-t}_\delta(\supp a)$. Since $\Phi^t_\delta$ is a family of 
diffeomorphisms depending smoothly on $t$, there exists a compact set 
$K\subset T^*X$ such that $\Phi^{-t}_\delta(\supp a)\subset K$ for all 
$t\in [-T,T]$.

Recall that the derivatives of the composition of a smooth functions 
$f\in C^\infty(\R^d)$ and $g\in C^\infty(\R^d;\R^d)$ are given by the 
generalized F\`aa di Bruno formula 
\begin{equation}\label{app:eg2}
	\partial^\alpha ( f\circ g ) = \sum_{\alpha_\ell, j} c_{\alpha,j} 
		(\partial_{x_{j_1}\cdots x_{j_m}} f)\circ g 
	\cdot
	\prod_{\ell=1}^m\partial^{\alpha_\ell}\psi_{j_\ell}, \quad \alpha \in \N^d, 
\end{equation}
where $c_{\alpha,j}$ are constants $j_\ell\in\{1,\dots,d\}$ and 
$\alpha_1,\dots,\alpha_m$ are multiindices whose sum is equal to 
$\alpha$. Since $a\in S_\beta^\comp(X)$, it follows from \eqref{eq:SchTyp6} 
and \eqref{app:eg2}  that $a^t\in S^\comp_\beta(X)$. 
\\
\\
2. Set $B(t):=U_\delta(-t)AU_\delta(t)$ and define a family of 
operators $B_k(t)$, $k\in \N$, iteratively, as follows: First, put
$B_0(t)=\Op_h(a^t) \in \Psi_\beta^\comp$. Then, 
\begin{equation}\label{app:eg3}
	[P_\delta,B_0(t)] = [P_0,B_0(t)] + \delta [Q,B_0(t)]
	=\frac{h}{i}\Op_h(\{p,a^t\}) 
	- \wh{E}_0(t) - \wt{E}_0(t),
\end{equation}
where $\wh{E}_0(t) = [P_0,B_0(t)]  - \frac{h}{i}\Op_h(\{p,a^t\})   \in h^{2(1-\beta)}\Psi_\beta^\comp$ and 
$\wt{E}_0(t) = \delta [Q,B_0(t)]\in \delta h^{2(1-2\beta)}\Psi_\beta^\comp$ by 
\eqref{app:PDO5.1} and \eqref{app:PDO7}. By \eqref{eq:DeltCondition}, 
it follows that $E_0(t):=\wh{E}_0(t) + \wt{E}_0(t)\in 
h^{2(1-\beta)}\Psi_\beta^\comp$. Notice that $E_0(t)$ depends smoothly 
on $t$ since this is the case for $B_0(t)$ and $a^t$. 
Let $e_0(t)\in h^{2(1-\beta)}S_\beta^\comp$ be the full symbol of $E_0(t)$. 
Since $\partial_t a^t = H_p a^t$, it follows from \eqref{app:eg3} that
\begin{equation}\label{app:eg3.1}
	hD_t B_0(t) = [P_h^\delta,B_0(t)] + E_0(t).
\end{equation}

Next define a sequence of symbols $e_k(t)\in h^{k(1-2\beta)+2(1-\beta)}S_\beta^\comp$, 
$c_{k+1}(t)\in h^{(k+1)(1-2\beta)}S_\beta^\comp$, $k\in\N$, depending smoothly 
on $t$, with $e_0(t)\in h^{2(1-\beta)}S_\beta^\comp$ as above and 
\begin{equation}\label{app:eg4}
	c_{k+1}(t) = \frac{i}{h}\int_0^t (\Phi^{t-s}_\delta)^*e_k(s)ds, 
\end{equation}
where $e_{k+1}(t)$ is the full symbol of $E_{k+1}(t)\in h^{(k+1)(1-2\beta)
+2(1-\beta)}\Psi_\beta^\comp$ defined via the relation 
\begin{equation}\label{app:eg5}
\begin{split}
	hD_t \Op_h(c_{k+1}(t)) 
	&= \frac{h}{i}\Op_h(\{p,c_{k+1}(t)\}) +\Op_h(e_k(t)) \\
	&= [P^\delta_h,\Op_h(c_{k+1}(t)) ] - E_{k+1}(t) +\Op_h(e_k(t)).
\end{split}
\end{equation}
Here, we used the fact that $hD_t c_{k+1}(t) = H_p c_{k+1}(t) + e_k(t)$ 
in the first line, and \eqref{app:PDO5.1}, \eqref{app:PDO7} in the last line. 
Put  
\begin{equation}\label{app:eg6}
	B_{k+1}(t) =B_k(t) - \Op_h(c_{k+1}(t)).
\end{equation}
Then, for each $k\in\N^*$ 
\begin{equation}\label{app:eg7}
	hD_t B_k(t) = [P^\delta_h,B_k(t)] + E_{k}(t) 
			+ h^\infty \Psi^{-\infty}.
\end{equation}

3. By Borel summation, there exists a $b(t)\in S^{-\infty}_\beta(X)$ such that
\begin{equation*}
	b(t) \sim \sum_{k\in\N}b_k(t), 
	\quad b_0(t) = a(t)\in S_\beta^\comp,~b_k(t) = -c_k(t) 
	\in  h^{k(1-2\beta)}S_\beta^\comp, k\geq1.
\end{equation*}
Hence, $\Op_h(b(0)) = A + A_1$,  with $A_1\in h^\infty \Psi^{-\infty}$, and 
\begin{equation*}
	hD_t\Op_h(b(t)) = [P^\delta_h,\Op_h(b(t))] + R(t), \quad R(t)\in  h^\infty \Psi^{-\infty}
\end{equation*}
By integration, we get that 
\begin{equation}\label{app:eg9}
\begin{split}
	U_\delta(-t)AU_\delta(t)-\Op_h(b(t)) 
	&= U_\delta(-t)\big( A-U_\delta(t)\Op_h(b(t))U_\delta(-t) \big)U_\delta(t)
	\\
	&= -U_\delta(-t)\left[\int_0^t hD_s \left(U_\delta(s)\Op_h(b(s))U_\delta(-s)\right) ds \right] U_\delta(t) 
	   - U_\delta(-t)A_1U_\delta(t)
	\\
	& =	-\int_0^t U_\delta(s-t)R(s) U_\delta(t-s) ds - U_\delta(-t)A_1U_\delta(t).
\end{split}
\end{equation}
Since $R(s)\in h^\infty \Psi^{-\infty}$, we see by \eqref{eq:SchTyp5} 
that $U_\delta(-t)AU_\delta(t)-\Op_h(b(t)) \in h^\infty \Psi^{-\infty}$. 
Applying the same argument to the term containing $A_1$, we conclude 
\eqref{app:eg0}. 
\par
The statement on the wave front set follows from \eqref{app:PDO6.1} and 
the observation that $\supp c_{k+1}(t) \subset \Phi^{-t}_\delta(\supp a)$, 
see \eqref{app:eg4}, as is true for $a(t)$. 
\end{proof}

\end{document}